\documentclass[oneside,english,reqno]{amsart}

\usepackage{multirow}
\usepackage{subfigure}

\usepackage[T1]{fontenc}
\usepackage[utf8]{inputenc}
\usepackage{babel}
\usepackage{amsfonts}
\usepackage{amssymb}
\usepackage{amsmath}

\usepackage{xcolor}
\usepackage{babel}
\usepackage{enumitem}
\usepackage{algorithm}
\usepackage{algorithmic}
\usepackage{comment}

\usepackage{amsthm}
\usepackage{enumitem}
\usepackage{graphicx}
\usepackage[unicode=true,pdfusetitle,
bookmarks=true,bookmarksnumbered=false,bookmarksopen=false,
breaklinks=false,pdfborder={0 0 1},backref=false,
hidelinks]
{hyperref}

\makeatletter
\numberwithin{equation}{section}
\numberwithin{figure}{section}
\usepackage[a4paper, 
 top=1cm, bottom=1.5cm, headsep=0.5cm]{geometry}
\usepackage[a4paper]{geometry}
\newtheorem{theorem}{Theorem}[section]

\newtheorem{condition}[theorem]{Condition}

\newtheorem{corollary}[theorem]{Corollary}

\newtheorem{assumption}[theorem]{Assumption}
\newtheorem{lemma}[theorem]{Lemma}

\newtheorem{definition}[theorem]{Definition}
\newtheorem{example}[theorem]{Example}

\newtheorem{proposition}[theorem]{Proposition}
\newtheorem{remark}[theorem]{Remark}

\makeatother

\newcommand{\minimize}[2]{\ensuremath{\underset
  {\substack{{#1}}}%
{\text{\rm minimize}}\;\;#2 }}
\newcommand{\maximize}[2]{\ensuremath{\underset
  {\substack{{#1}}}%
{\text{\rm maximize}}\;\;#2 }}
\newcommand{\A}{ \mathcal{A}}
\newcommand{\OAP}{ \mathcal{O}}
\newcommand{\lev}[1]{{\ensuremath{{{{\operatorname{lev}}}_{\leq #1}}\,}}}

\usepackage{algorithm}

\usepackage{float}

\newcommand{\X}{\mathbb{R}^n}
\usepackage{algorithmic} 

\newcommand{\R}{\mathbb{R}}

\begin{document}

\title{Outer Approximation Scheme for Weakly Convex Constrained Optimization Problems}


\author{E. Bednarczuk$^{1,2}$, G. Bruccola$^1$, J.-Ch. Pesquet$^3$, K. Rutkowski$^{1,4}$
}

\thanks{$^1$ Systems Research Institute of the Polish Academy of Sciences. }
\thanks{$^2$ Warsaw University of Technology.}
\thanks{$^3$ Centrale-
Supélec, CVN, Inria, University Paris-Saclay, 91190 Gif-sur-Yvette, France.}
\thanks{$^4$ Cardinal Stefan Wyszy\'nski University, 01-815 Warsaw, Dewajtis 5.	}


\maketitle

\begin{abstract}

Outer approximation methods have long been employed to tackle a variety of optimization problems, including linear programming, in the 1960s, and continue to be effective for solving variational inequalities, general convex problems, as well as mixed-integer linear, and nonlinear programming problems.

In this work, we introduce a novel outer approximation scheme specifically designed for solving weakly convex constrained optimization problems. The key idea lies in utilizing quadratic cuts, rather than the traditional linear cuts, and solving an outer approximation problem at each iteration in the form of a Quadratically Constrained Quadratic Programming (QCQP) problem.

The primary result demonstrated in this work is that every convergent subsequence generated by the proposed outer approximation scheme converges to a global minimizer of the general weakly convex optimization problem under consideration. To enhance the practical implementation of this method, we also propose two variants of the algorithm.

The efficacy of our approach is illustrated through its application to two distinct problems: the circular packing problem and the Neyman-Pearson classification problem, both of which are reformulated within the framework of weakly convex constrained optimization.

\keywords{Key words: Weak convexity, QCQP, Outer approximation, Quadratic cuts}

\end{abstract}

\section{Introduction}

We consider the following constrained optimization problem:
\begin{equation}
    \label{prob: P0}
    \tag{P0}
    \minimize{x\in \mathbb{R}^{n-1}}\,F(x)\quad \text{s.t. }\quad g_i(x)\le0,\ \ i\in \{2,\ldots,m\},
\end{equation}
where $n\in \mathbb{N}\setminus \{0,1\}$, $F\colon \mathbb{R}^{n-1} \to (-\infty, +\infty]$ is a proper lower-semicontinuous (lsc) $\beta_1$-weakly convex function with $\beta_1\in \mathbb{R}_+$ and, for every
$i\in \{2,\ldots,m\}$, $g_i\colon \mathbb{R}^{n-1} \to (-\infty, +\infty]$ is a proper lsc $\beta_i$-weakly convex function with
$\beta_i\in \mathbb{R}_+$.
This problem can be reformulated as
\begin{equation}
    \label{P_min}
    \minimize{x\in S}\,F(x)
\end{equation}
where $S = \bigcap\limits_{i=2}^m \text{lev}_{\le0}\,g_i$ 
and, for every $i\in \{2,\ldots,m\}$, $\text{lev}_{\le 0}\,g_i = \{x\in \mathbb{R}^{n-1}\mid g_i(x)\le 0\}$ is the lower-level set of $g_i$ at height 0.

Problem \eqref{P_min} defines a superclass of convex constrained optimization problems and many optimization problems can be cast in the form of  \eqref{prob: P0}. 

Let us denote the inner product of the space $\mathbb{R}^n$ by $\langle\cdot,\cdot\rangle$  and let $\|\cdot\|$ be the corresponding norm.
We show (see Theorem \ref{th: minnoncon} below) that, under mild assumptions, \eqref{prob: P0} can be rewritten as follows:
\begin{equation}
    \label{problem}
    \tag{P}
    \minimize{x\in \A} \ \ J(x)=\frac{1}{2}\|x-z\|^2,
\end{equation}
where $z\in\mathbb{R}^n$ and
\begin{equation}\label{e:defA}
\A:=\{x\in\mathbb{R}^{n}\mid (\forall i \in I= \{1,\ldots, m\})\, f_{i}(x)\le 0\},
\end{equation}
for every $i\in I$, function $f_{i}\colon\mathbb{R}^n\rightarrow(-\infty, +\infty]$ is proper, lsc, and $\rho_i$-weakly convex with $\rho_i \in \mathbb{R}_+$. Throughout the paper, we will analyze problem \eqref{problem}.

Unconstrained minimization problems with weakly convex objective have been investigated  within different frameworks, {\em e.g. by 
\cite{jourani1996subdifferentiability, pallaschke2013foundations, rubinov2013abstract, singer1997abstract}, }
and a number of splitting proximal algorithms have been proposed {\em e.g.} in \cite{Atenas2023Unified,bohm2021variable,khanh2023inexactproximalmethodsweakly}. 

Let $J^*$ be the optimal value of \eqref{problem}, which is finite provided that $\A$ is nonempty.
If, for every $i\in I$, $f_i$ is continuous on a lower level set 
\begin{equation}
\lev{\alpha}J =\{x\in\mathbb{R}^n \mid J(x)\le \alpha\},
\end{equation}
and the height $\alpha> J^*$, then we will show that \eqref{problem} can be solved by a new outer approximation method called \textit{Cutting spheres method}.
The idea of the Cutting spheres method is to construct a sequence of elements which approximate the solution to problem \eqref{problem} by relaxing the constraints in a suitable iterative manner.

At each iteration $k$, this method solves a quadratically  constrained (possibly nonconvex)  optimization problem, which is an outer approximation to the original problem \eqref{problem}, in the sense that
set $\mathcal{A}$ is replaced by a superset $\OAP_k$, i.e. 
$$ \A\subset \OAP_k,$$
where $\OAP_k$ is the solution set of a finite number of possibly nonconvex quadratic inequalities.
\subsection{Related works}

Consider Problem \ref{P_min}
where $F\in \Gamma_0(\mathbb{R}^n)$ and $S\subset \mathbb{R}^n$ is convex.
Outer approximation methods for convex optimization, see e.g. \cite{combettes2000strong,combettes2004image,drori2016optimal}, generalize the Cutting plane method introduced in \cite{kelley1960cutting}. 
The main idea of the Cutting plane method is to generate, at each iteration $k$, hyperplanes which separate the current iteration $x_k$ from the feasible set 
$S$.
In this way, the halfspaces generate a polyhedral outer approximation set of $C$ , which does not contain the current iteration $x_k$. 
By minimizing the objective function $F$
over the outer approximation set, we find the new iterate and we move forward in the direction of 
$S$.
We can say that the hyperplanes \textit{cut} the current iteration $x_k$.

It is important to notice the connection between the Cutting plane method and monotone operator theory. A \textit{cutter} is an operator 
$T:\X\rightarrow\X$ 
such that
\begin{equation}
    \label{def: cutter}
    (\forall\,x\in \X)\quad
    \text{Fix }T \subseteq h(x,Tx):=\{ u\in \X\ |\ \langle u-Tx, x-Tx\rangle \le 0\},
\end{equation}
where $\text{Fix }T:=\{x\in \X\ |\ Tx=x  \}$. Then, for a given $k\in \mathbb{N}$, the bounding hyperplane of $h(x_k,Tx_k)$ \textit{cuts} the space in two half-spaces, one which contains $\text{Fix }T$ and the other which contains $x_k\not\in \text{Fix }T$.
The name cutter operators was suggested by \cite{cegielski2011opial}, but the class of operators satisfying \eqref{def: cutter} was introduced, ante litteram, by \cite{combettes2001quasi} and also studied, for example, in \cite{censor2009split,zaknoon2003algorithmic} under different names.
The goal of the papers cited above is to propose an algorithm that, given a sequence of cutter operators $(T_k)_{k\in \mathbb{N}}$, where $(\forall k\in \mathbb{N})$ $T_k:\X\rightarrow\X$, converges to a common fixed point. 

Finally, also bundle methods, e.g. \cite{kiwiel1990proximity, lukvsan1999globally},  can be seen as generalizations of the Cutting plane algorithm proposed in \cite{kelley1960cutting} for convex problems. 

\subsection{Outer approximation in nonconvex settings and contribution}

Several outer approximation methods deal with nonconvex constrained optimization problems.  For example,  \cite{bienstock2019intersection,tuy1988global,yamada2010outer}, minimize the objective over the intersection $C\cap S$ of a convex set $C$ and a non convex set $S$ with a specific structure. The core idea of these three papers is to move in the direction of $C$ with a classic convex outer approximation method. In order to converge to a point in $C\cap S$, \textit{ad hoc} techniques are employed. 

On the other hand, the core idea of our method differs from the other outer approximation methods found in the literature, and it is sketched in the following.  Within our settings it does not  always exist a hyperplane separating the feasible set $\A$ of \eqref{problem} to a point $x\in \mathbb{R}^n\setminus\A $. 

We thus replace hyperplanes with spheres by using the tools provided by abstract convexity. 
In the setting of abstract convexity (which include weak convexity) for unconstrained optimization,  a generalized Cutting plane method can be found in \cite{pallaschke2013foundations, rubinov2013abstract}, where successive underestimations of the abstract $\Phi$-convex objective function are performed by using the $\Phi$-subgradient inequality, see \eqref{subgradient}, instead of the convex subgradient. 


 Based on $\Phi$-subgradients to the weakly convex functions $(f_i)_{1\le i\le m}$ describing the constraint set $\A$ of Problem \eqref{problem}, we construct an outer approximation to $\A$ in order to generate a new outer approximation problem at each iteration.

Geometrically, the resulting outer approximation to $\A$, is  the complement of the union of open balls.
The interiors of the balls the current iteration belongs to will be cut out, while $\A$ lies in the complement of the union of the ball interiors.


At each iteration, we solve a quadratic  polynomial  optimization problem. We call our method \textit{Cutting spheres algorithm} \ref{alg: alg1}.
Our main result is that 
every sequence generated by the Cutting spheres algorithm \ref{alg: alg1} contains a convergent subsequence and every
 converging subsequence converges to the global optimal solution of \eqref{problem}, see Theorem \ref{th: convergence}.

The drawback of \textit{Cutting spheres algorithm} \ref{alg: alg1} is that the number of constraints of the outer approximation problem increases with the number  of iterations.
Moreover, the outer approximation problem that we solve at every iteration is NP-Hard.
This is the reason why we propose another algorithm, \textit{Cutting spheres algorithm with warm restarts} \ref{alg: warm_rest}, which introduces complexity bound 
at the expense of weakening some convergence properties of Cutting spheres algorithm \ref{alg: alg1}.
 Finally, we propose a variant of the  \textit{Cutting spheres algorithm with warm restarts} \ref{alg: warm_rest}, called  \textit{Inexact cutting spheres algorithm} \ref{alg: inexact}, which possess  some computational advantages related to constraint linearization.


Up to our knowledge, Cutting spheres algorithm \ref{alg: alg1} is the first attempt to globally solve Problem \eqref{prob: P0} with convergence guarantees.
Studies involving weakly convex functions have appeared in various optimization
problems 
in a rapidly growing number, see \textit{e.g.}\ \cite{Atenas2023Unified,bohm2021variable, davis2019stochastic,  zhang2019fundamental}.  Moreover, weakly convex problems also have various applications. They occur, for example, in signal processing \cite{liu2019doa}, image processing  \cite{bayram2015convergence,mollenhoff2015primal,nikolova1998estimation,selesnick2020non}, and machine learning \cite{rafique2022weakly}.


\subsection{Motivation}

    Our motivation is to provide a convergent algorithm for finding global solutions of the constrained optimization problem \eqref{prob: P0}, which makes a direct use of weak convexity of the functions defining the feasible set $\mathcal{A}$.

{\color{black} In contrast to the convex case, a conceivable  passage from the constrained problems \eqref{prob: P0} and \eqref{problem} to unconstrained ones  by adding the indicator function of the set ${\A}$ to the objective $F$ is not easily achievable. This is  mainly due to the fact that the relationships between the weak convexity of a function, the weak convexity of its level sets and/or of the indicator function of ${\A}$  are still an active research topic,  and different concepts of a weakly convex set exist (see \cite{BALASHOV2010113,8955328}). 

Weakly convex constrained problems can be approached by quadratic regularization techniques as described in \cite{boob2023stochastic,ma2020quadratically}
where the authors solve a sequence of strongly convex problems. However, these papers provide only convergence to nearly stationary and $\varepsilon$-feasible solutions.

}

\subsection{Organization of the paper}
The rest of this paper is organized as follows
\begin{itemize}
\item[Section 2] We reformulate Problem \eqref{prob: P0} as \eqref{problem}.
    \item[Section 3]
    We define the outer approximation sets for feasible set $\mathcal{A}$. 
\item[Section 4] We introduce Cutting spheres algorithm \ref{alg: alg1} and investigate its basic properties. 
\item[Section 5]   We provide a convergence theorem for the Cutting spheres algorithm \ref{alg: alg1} (Theorem \ref{th: convergence}). We prove that there exists a converging subsequence and every cluster point of the sequence is a global solution to \eqref{problem}.
\item[Section 6] We introduce the Cutting spheres  with warm restarts algorithm \ref{alg: warm_rest}. This new algorithm addresses the problem of controlling the number of constraints at every iteration.
\item[Section 7] We discuss the quadratic approximation subproblem \eqref{outerproblem}.
\item[Section 8] We discuss the Inexact cutting spheres algorithm \ref{alg: inexact}, and we provide a convergence proof to a global $\varepsilon$-solution.
\item[Section 9] We present the numerical experiment related to the packing problem.
\item[Section 10]  We present the numerical experiment related to multiclass Neyman-Pearson classification.
\end{itemize}

\section{ Reformulating Problem (\texttt{P0}) as (\texttt{P})}
\label{chap: reformulation}
Given two sets $\mathcal{K}$ and $\mathcal{L}$, $|\mathcal{K}|$ denotes the cardinality of $\mathcal{K}$, while $\mathcal{K}\setminus \mathcal{L}$ denotes the set
$
\{k\in \mathcal{K}\mid k\not\in \mathcal{L}\}.
$
Recall that a function $f\colon \mathbb{R}^n\rightarrow ]-\infty, +\infty]$ is $\rho$-weakly convex with $\rho \in \mathbb{R}_+$ if $f+\frac{\rho}{2}\|\cdot\|^2$ is convex.

\begin{theorem}
\label{th: minnoncon}
Assume that Problem \eqref{P_min} has a solution. Let $(\rho,\eta)\in]0,+\infty[ \times \mathbb{R}$ be such that
\begin{equation}
    \label{set: D}
    D = \left\{\widehat{x}\in  \operatorname{Argmin}_{x\in S}F(x) \mid F(\widehat{x}) + \eta \ge \frac{\rho}{2}\|\widehat{x}\|^2\right\}\neq\varnothing. 
\end{equation}
For every $i \in \{1,\ldots, m\}$ and $y = (\overline{y}, y_{n}) \in \mathbb{R}^{n}$, let
\begin{equation}
  \widetilde{f}_i(y)=
  \begin{cases}
  F(\overline{y}) + \eta - \frac{\rho}{2}\|y\|^2 & \text{ if }i=1   \\
  g_i(\overline{y}) &\text{ otherwise.}
  \end{cases}
\end{equation}
Then $\widehat{x} \in D$ if and only if there exists
$\widehat{y}_{n}\in \mathbb{R}$ such that
$\widehat{y} = (\widehat{x}, \widehat{y}_{n})$ is a solution to
\begin{equation}
\label{problem2}
   \minimize{y\in\mathbb{R}^{n}}\,\|y\|^2\quad \text{s.t.}\quad
   (\forall i \in \{1,\ldots,m\})\;
   \widetilde{f}_i(y)\le 0.
\end{equation}
\end{theorem}
\begin{proof}
By introducing a slack variable, Problem \eqref{P_min} is equivalent to
\begin{equation}
\label{mincon0.6}
\minimize{x\in\mathbb{R}^{n-1}, x_{n}\in\,\mathbb{R}}\,x_{n}\
\quad \text{s.t. }\quad
\begin{cases}
F(x)+\eta\le\frac{\rho}{2}x_{n},\\
(\forall i \in \{2,\ldots,m\})\; 
g_i(x)\le0.
\end{cases}
\end{equation}
Since \eqref{set: D} holds, the feasible solutions $(x,x_n)$ to \eqref{mincon0.6}, where $x\in D$, satisfy the additional constraint
\begin{equation}
    \label{mincon0.7}
    \|x\|^2\le x_{n}.
\end{equation}
This shows that we can make the change of variables
\begin{align}
    &x = \overline{y}\nonumber\\
    &x_{n} = \|\overline{y}\|^2 + y^2_{n}\nonumber\\
    &y=(\overline{y},y_{n}), 
\end{align}
so leading to Problem \eqref{problem2}.
\end{proof}
In view of Theorem \ref{th: minnoncon}, in the sequel we focus on solving only problem \eqref{problem}.

\begin{remark}
Notice that \eqref{problem2} is equivalent to \eqref{problem}.
If $F$ and $(g_i)_{2\le i\le m}$ are weakly convex (e.g., twice differentiable functions with bounded Hessian), then the functions $(\widetilde{f}_i)_{1\le i\le m}$ are also weakly convex. However, if Problem \eqref{prob: P0}
is convex (i.e., $\beta_1 = \ldots = \beta_m = 0$), Problem \eqref{problem2} has a weakly convex constraint (namely,  $\widetilde{f}_1(y)\le0$).
\end{remark} 

\begin{remark}
 If $F$ is $\beta_1$-weakly convex and, for every $i\in\{2,\ldots,m\}$, $g_i$
is $\beta_i$-weakly convex, then there exists convex functions $H$ and $(h_i)_{1\le i\le m}$ such that, for every $x\in\mathbb{R}^n$,
\begin{equation}
    \label{mincon_weakconv}
    \begin{aligned}
    F(x)=H(x)-\frac{\beta_1}{2}\|x\|^2\\
    (\forall\,i\in\{2,\ldots,m\})\ \ g_i(x)=h_i(x)-\frac{\beta_i}{2}\|x\|^2.
    \end{aligned}
\end{equation}
\end{remark}

Therefore, Problem \eqref{problem} can be rewritten as
\begin{equation}
    \minimize{y=(\overline{y},y_{n})\in\mathbb{R}^{n},\sigma\in[0,+\infty[}\sigma\quad \text{s.t. }\quad \left\{
    \begin{aligned}
        &H(\overline{y})+\frac{\beta_1}{2} y_{n}^2+\eta\le \frac{\rho+\beta_1}{2}\sigma\\
        &(\forall\,i\in\{2,\ldots,m\})\ \ h_i(\overline{y})+\frac{\beta_i}{2}y^2_{n}\le \frac{\beta_i}{2}\sigma\\
        & \sigma=\|y\|^2.
        \end{aligned}
\right.  
\end{equation}
We deduce that a convex relaxation of Problem \eqref{problem} is given by
\begin{equation}
\label{mincon_0.14}
    \minimize{y\in\mathbb{R}^{n},\sigma\in[0,+\infty[}\sigma\quad \text{s.t. }\quad \left\{
    \begin{aligned}
        &H(\overline{y})+\frac{\beta_1}{2} y_{n}^2+\eta\le \frac{\rho+\beta_1}{2}\sigma\\
        &(\forall\,i\in\{2,\ldots,m\})\ \ h_i(\overline{y})+\frac{\beta_i}{2}y^2_{n}\le \frac{\beta_i}{2}\sigma\\
        & \|y\|^2\le \sigma.
        \end{aligned}
\right.  
\end{equation}
If $(\widehat{x},\widehat{y}_{n},\widehat{\sigma})$ is a solution to Problem \eqref{mincon_0.14}, then $(\widehat{x},0,\widehat{\sigma})$ is a solution too. In turn, in the convex case, the pairs of the form $(\widehat{x}, 0)$ are the solutions to Problem \eqref{mincon0.6}-\eqref{mincon0.7}.
(set $x_{n} = \sigma$). So, when $F$ and $(g_i)_{2\le i\le m}$ are convex, the relaxation \eqref{mincon_0.14} is exact (in sense of having the same solutions).

\section{About the constraint set}
\subsection{Weakly convex constraint functions}
\label{subsec: Weakly convex constraint functions}

We define $\text{dom }f_i:=\{x\in \mathbb{R}^n\ |\ f_i(x)<+\infty\}$. We say that the function $f:\ \mathbb{R}^n\rightarrow (-\infty,+\infty]$ is \textit{proper} when $\text{dom }f\neq \varnothing$.

Consider the constraint set $\A$ in \eqref{problem}  and recall that, for every $i\in I$,
function $f_{i}\colon\mathbb{R}^n \to (-\infty,+\infty]$ has been assumed  proper,
lower semicontinuous, and $\rho_i$-weakly convex on $\mathbb{R}^n$, i.e.,  for each function $f_i$ there exists a constant $\rho_i\ge 0$ such that function 
$$ \check{f}_i=f_i+\frac{\rho_i}{2}\|\cdot\|^2$$
is convex. 
Functions $(f_{i})_{i\in I}$ can be represented (see e.g. \cite[Proposition 1.1.3(f)]{cannarsa2004semiconcave})  as the pointwise supremum of functions from the class 
$$
\Phi^i_{\rm lsc}=\{\varphi\colon \mathbb{R}^n\rightarrow\mathbb{R}\mid
(\forall x \in \mathbb{R}^n)\,\varphi(x)=-a\|x\|^{2}+\langle b,x\rangle+\Tilde{c}, \ a\ge\frac{\rho_i}{2}, b\in\mathbb{R}^{n}, \Tilde{c}\in\mathbb{R}\},
$$
i.e., for every $i\in I$ and $x\in \mathbb{R}^n$,
\begin{equation}\label{e:conja}
f_{i}(x)=\sup_{\varphi\in \Phi^i_{\rm lsc}} \{\varphi(x)\mid \varphi\le f_{i}\},
\end{equation}
where $\varphi$ are functions in $\Phi^i_{\rm lsc}$ minorizing $f_{i}$ ($\varphi\le f_{i}$ means that, for every $x\in \mathbb{R}^n$, $\varphi(x)\le f_{i}(x)$).

The notation $\Phi_{\rm lsc}$ comes from the fact \cite[Proposition 6.3]{rubinov2013abstract}
that the class of lower-semicontinuous functions minorized by an element in  $\Phi_{\rm lsc}$ coincides with the class of $\Phi_{\rm lsc}$-convex functions.  
    We say the function $f_i$, $i\in I$ 
is $\Phi^i_{\rm lsc}$-convex (or abstract convex with respect to $\Phi^i_{\rm lsc}$) if \eqref{e:conja} holds.

We will also define the set
$$
\Phi_{\rm lsc}^{i,0}=\{\varphi\in \Phi_{\rm lsc}^{i}\mid
\varphi(0) = 0\}.
$$


Let $\widetilde{x}\in \text{dom }f_i$.
An element $\varphi\in \Phi_{\rm lsc}^{i,0}$ 
is an {\em abstract subgradient or $\Phi_{lsc}^i$-subgradient} of the function $f_{i}$ at the point $\widetilde{x}$ if
\begin{equation} 
\label{subgradient}
(\forall x \in \mathbb{R}^n)\quad
f_{i}(x)-f_{i}(\widetilde{x})\ge\varphi(x)-\varphi(\widetilde{x}).
\end{equation}
We will then write $\varphi\in\partial_{\rm lsc}f_{i}(\widetilde{x})$.
Note that we can rewrite \eqref{subgradient} as 
\begin{equation}
    \label{convex subgradient}
    \begin{aligned}
    (\forall x \in \mathbb{R}^n)\quad&\Tilde{f}_i(x)- \Tilde{f}_i(\widetilde{x})\ge \langle b_i, x-\widetilde{x}\rangle
    \end{aligned}
\end{equation}
and $b_i\in\partial \Tilde{f}_i(\widetilde{x})$, where $\partial\Tilde{f}_i(\widetilde{x})$ is the (Moreau) subdifferential in the sense of convex analysis of $\Tilde{f}_i=f_i+a_i\|\cdot\|^2$ at $\widetilde{x}$ with $a_i\ge \rho_i/2$, see e.g., \cite[Definition 16.1]{bauschke2017convex}.
This can be summarized by the following property.
\begin{lemma}
\label{lem: quad form constr}
For every $i\in I$, let $f_i$ be a proper lsc and $\frac{\rho_i}{2}$-weakly convex function
and let $\widetilde{x} \in \text{dom }f_i$. 
If $\partial_{\rm lsc}f_i(\widetilde{x})\neq \varnothing$ and $\varphi\in \partial_{\rm lsc}f_i(\widetilde{x})$, then 
$f_i$ 
has a quadratic tangent minorant $L_i(\varphi,\widetilde{x})$, 
that is
\begin{equation}
    \label{quad form constr}
    (\forall x \in \mathbb{R}^n)\quad f_i(x) \geq 
    L_i(\varphi,\widetilde{x})(x) =\varphi(x)+c_i(\widetilde{x}),
\end{equation}
where $c_i(\widetilde{x}) = f_i(\widetilde{x})-\varphi(\widetilde{x})$.
\end{lemma}

\subsection{Outer approximation to set $\A$}

 For every $x\in \mathbb{R}^n$, let
\begin{equation}\label{e:defIxell}
    I(x) = \{i \in I \mid f_{i}(x)>0\}
\end{equation}
be the index set of violated constraints at $x$.
 \begin{definition}
 \label{def: out approx}
Let $k\in \mathbb{N}$ and $(x_0,x_1,\dots,x_k)\in (\mathbb{R}^n)^k$, be such that, for every $\ell \in \{0,\ldots,k\}$,
$I(x_\ell)\neq \varnothing$ and, for every $i\in I(x_\ell)$, there exists $\varphi_i^\ell\in 
\partial_{\rm lsc}f_i(x_\ell)$.
The {\em outer approximation set} $\mathcal{O}_k$ is defined as
\begin{equation}
    \mathcal{O}_k := \{ x \in \mathbb{R}^n \mid 
    (\forall \ell \in \{0,\ldots,k\})(\forall i \in I(x_\ell))\;
    L_i(\varphi_i^\ell, x_\ell)(x) \le 0\}.
\end{equation}
We will refer to the constraint $L_i(\varphi_i^k,x_k)(\cdot)\le 0$ as a Quadratic cut.
 \end{definition}
 Note that $\mathcal{O}_k$ with $k\ge 1$ defined above can be built recursively as
 \begin{equation}\label{e:recurOk}
 \mathcal{O}_k = 
 \{ x \in \mathbb{R}^n \mid 
    (\forall i \in I(x_k)) 
    L_i(\varphi_i^k, x_k)(x) \le 0\} \cap \mathcal{O}_{k-1}.
\end{equation}
The outer approximation terminology is justified by the following result which
is  a straightforward consequence of Lemma \ref{lem: quad form constr}.
\begin{lemma}
\label{lem: out set}
Consider the feasible set $\mathcal{A}$ of Problem  \eqref{problem}. 
Under the assumptions of Definition \ref{def: out approx}, the following properties are satisfied.

\begin{enumerate}
 \item\label{lem: out seti} For every $i\in  I(x_k)$, $f_i(x_k)=L{_i}(\varphi^k_i,x_{k})(x_k)>0$.
\item\label{lem: out setii} $\mathcal{A}\subset \mathcal{O}_k$.
\end{enumerate}
\end{lemma}

\subsection{Geometric interpretation of the outer approximation constraints set $\mathcal{O}_k$
}
\label{subsec: geometric}

The following proposition provides a geometric interpretation of the Quadratic cut $L{_j}(\varphi^k_j,x_{k})(\cdot)\le 0$, $j\in I(x_k)$, $k\in \mathbb{N}$.

\begin{lemma}
\label{lem: geo_sphere}
Under the assumptions of Definition \ref{def: out approx},
given an index $j\in I(x_k)$,
the constraint
$$
L_j(\varphi_{j}^{k}, x_{k})(\cdot)\le 0
$$
with $\varphi_{j}^{k}\colon x \mapsto -a_{j} \|x\|^2+\langle b_{j,k}, x \rangle$,
$a_j > 0$, and $b_{j,k}\in \mathbb{R}$,
describes the complement of a ball centered at $b_{j,k}/(2a_j)$ with radius 
\begin{equation}
    \label{eq:sphere_radius}
    r_k=\sqrt{\left\|x_k-\frac{b_{j,k}}{2a_j}\right\|^2+\frac{f_j(x_k)}{a_j}}.
\end{equation}
\end{lemma}
\begin{proof}
By construction, $L_j(\varphi_{j}^{k}, x_{k})(\cdot)> 0$ is an
an open ball
centered in $b_{j,k}/(2a_j)$, 
hence ${L_j(\varphi_{j}^{k}, x_{k})(x)= 0}$ is equivalent to
\begin{equation}
    \label{eq: open_ball_constr2}
    \left\|x-\frac{b_{j,k}}{2a_j}\right\|=r_k,
\end{equation}
for some $r_k\in \mathbb{R}_+$.
We calculate the radius $r_k$:
\begin{align*}
    &L_j(\varphi_{j}^{k}, x_{k})(x)= 0\\
    \Leftrightarrow \quad &\varphi_j^k(x)-\varphi_j^k(x_k)+f_j(x_k)=0\\
    \Leftrightarrow \quad  &-a_j\|x\|^2+b_{j,k}^\top x+a_j\|x_k\|^2-b_{j,k}^\top x_k+f_j(x_k)=0\\
    \Leftrightarrow \quad &\left\|x-\frac{b_{j,k}}{2a_j}\right\|^2=\left\|x_k-\frac{b_{j,k}}{2a_j}\right\|^2+\frac{f_j(x_k)}{a_j},
\end{align*}
which yields \eqref{eq:sphere_radius}
(by assumption, $f_j(x_k)>0$).
\end{proof}

\section{Cutting spheres Algorithm}

In the sequel, we make the following assumption.
\begin{assumption}\  \label{a:compact}
\begin{enumerate}[start=1,label=(\roman*)]
    \item\label{a:compacti} $\mathcal{A}\neq \varnothing$
\item\label{a:compactii} There exists $\alpha>J^*$ such that
the constraint functions $(f_i)_{i\in I}$ are continuous on an open set containing $E:=\lev{ \alpha}J$.
\end{enumerate}
\end{assumption}
Note that, since $J$ is coercive and $\mathcal{A}$ is closed,
Assumption \ref{a:compact}\ref{a:compacti} guarantees the existence of a solution to Problem \eqref{problem}.

Consider now the following scheme for the outer approximation method.

\begin{algorithm}[H]
\setcounter{algorithm}{0}
\caption{Cutting spheres algorithm}
\label{alg: alg1}
\begin{algorithmic} 
\ENSURE $x_{0}=z$, $k=0 $, $\mathcal{O}_{-1}=\mathbb{R}^n$. 
\LOOP
\IF{$(\exists\,i\in I)$ $f_i(x_{k}) > 0$}
\STATE Build $\mathcal{O}_{k}$ as in Definition \ref{def: out approx}.
\STATE Choose $x_{k+1} \in 
\arg\min
_{x\in \mathcal{O}_{k}} \;\|x-z\|^{2}$ \qquad\% solution of outer approximation problem
\ELSE 
\STATE $\mathcal{O}_{k}=\mathcal{O}_{k-1}$
\STATE $x_{k+1}=x_k$
\ENDIF
\STATE $k\leftarrow k+1$
\ENDLOOP
\end{algorithmic}
\end{algorithm}

\begin{lemma}\label{le:welldef}
Under Assumption \ref{a:compact},
Algorithm \ref{alg: alg1} is well-defined and generates a sequence $(x_k)_{k\in \mathbb{N}}$ in $E$.
\end{lemma}
\begin{proof}
We will proceed by induction to show that, for every $k\in \mathbb{N}$,
$\mathcal{O}_{k-1}$  can be constructed and $x_{k}\in E$. The property obviously
holds for $k=0$ since $J(z) = 0$. 
Assume that $x_{k}\in E$ and $\mathcal{O}_{k-1}$  has been constructed for some $k\in \mathbb{N}$. We will show that $x_{k+1}\in E$ and $\mathcal{O}_{k}$  can be constructed. 

If $(\forall i \in I)$ $f_i(x_k) \le 0$, the result follows from the fact that $\mathcal{O}_{k}= \mathcal{O}_{k-1}$
and $x_{k+1} = x_{k}$. On the other hand, if the constraints are violated
for indices in the nonempty subset ${I}(x_k)$ of $I$, building $\mathcal{O}_k$ necessitates that $(\forall i \in {I}(x_k))$ 
$\partial_{\rm lsc}f_i(x_k) \neq \varnothing$. 
We know that $x_k\in E$ and, by assumption, $(f_i)_{i\in I(x_k)}$  are continuous on $E$.
This implies that, for every $i\in I(x_k)$ and $a_i\ge \rho_i/2$,
$\Tilde{f}_i =f_i+a_i\|\cdot\|^2$ is a convex subdifferentiable function on $E$
(see \cite[Proposition 16.17(ii)]{bauschke2017convex}), and consequently $\partial_{\rm lsc}f_i(x_k)\neq \varnothing$.
According to Lemma \ref{lem: out set}\eqref{lem: out setii}, $\varnothing \neq \A \subset \mathcal{O}_k$.
As $x \mapsto \|x-z\|^2$ is a coercive continuous function, the existence
of a minimizer $x_{k+1}$ of this function is guaranteed on the closed nonempty set $\mathcal{O}_k$
and $J(x_{k+1}) \le J^* < \alpha$, which shows that $x_{k+1} \in E$.
\end{proof}

The outer approximation sets introduced in Definition \ref{def: out approx}
produce a cumulative basis for Algorithm~\ref{alg: alg1} in the sense of  \cite[section 5.1]{combettes2000strong}, i.e. 
they guarantee that, for every iteration $k$,
\begin{equation}
    \label{eq: inclusions}
\A\subseteq \OAP_k\subseteq \OAP_{k-1} \subseteq \ldots \subseteq \OAP_0.
\end{equation}

We now establish the following necessary and sufficient condition for the algorithm to deliver a global optimal solution in a finite number of iterations.
\begin{lemma}
\label{lem: alg1}
Let $(x_k)_{k\in\mathbb{N}}$ be the sequence generated by Algorithm~\ref{alg: alg1}. Let $k\in \mathbb{N}$, and let $I(x_k)$ be the index set of violated constraints as defined by \eqref{e:defIxell}.
Then, 
$I(x_k)=\varnothing$ if and only if $J^*=\min_{x\in \mathcal{O}_{k-1}}\; J(x)$.

\end{lemma}   
 \begin{proof}
 Notice that
      \begin{equation}
          \label{x_k def}
          J(x_k)=\min_{x\in \mathcal{O}_{k-1}}\ \ J(x)
      \end{equation}
       and, by Lemma \ref{lem: out set}\eqref{lem: out setii}, $\A\subset \mathcal{O}_{k-1}$. Hence, $J(x_k)\le J^*$.
       We have the following equivalences:
       \begin{align}
           & I(x_k)  = \varnothing \nonumber\\
           \Leftrightarrow\quad  & x_k \in \A \nonumber\\
           \Leftrightarrow\quad & J(x_k)\ge J^* \nonumber\\
           \Leftrightarrow\quad & J(x_k)= J^*.
       \end{align}
 \end{proof}

\section{Convergence of Cutting spheres algorithm \ref{alg: alg1}}

The following preliminary result will be needed in our main convergence theorem.
\begin{theorem}
\label{th: subdifferential convergence}
Suppose that Assumption \ref{a:compact} holds.

Let 
$(x_\ell)_{\ell\in \mathbb{N}}$ be a sequence in $\mathbb{R}^n$
converging to
$\overline{x}\in \mathbb{R}^n$.
Let $i\in I$
and let $(b_i^\ell)_{\ell\in \mathbb{N}}$ be
such that
$$
(\forall\,\ell\in \mathbb{N})\ \ b_i^\ell\in\partial\,\Tilde{f}_i(x_\ell),
$$
where $\Tilde{f}_i =f_i+a_i\|\cdot\|^2$ with $a_i\ge \rho_i/2$, and $\partial \Tilde{f}_i$ is its
Moreau subdifferential. 
\begin{enumerate}
    \item\label{th: subdifferential convergencei} If $\overline{x}\in \operatorname{int}\operatorname{dom}\Tilde{f}_i$    ($\text{int}$ denotes the interior), there exists a convergent subsequence of $(b_i^\ell)_{\ell\in \mathbb{N}}$, and every convergent subsequence $(b_i^{\ell_k})_{k\in \mathbb{N}}$ converges
    to a subgradient of $\Tilde{f}_i$ at $\overline{x}$.
\item In addition, if $\Tilde{f}_i$ is differentiable at $\overline{x}$, then
$b_i^\ell\rightarrow \nabla \Tilde{f}_i(\overline{x})$.
\end{enumerate}
\end{theorem}
\begin{proof}\ 
\begin{enumerate}
    \item The function $\Tilde{f}_i$ is continuous on $\text{int dom }\Tilde{f}_i$, see \cite[Corollary 8.39]{bauschke2017convex}.
    We deduce from \cite[Proposition 16.17(iii)]{bauschke2017convex},
    that there exists a closed ball $B(\overline{x},\rho)$ with center
    $\overline{x}$ and radius $\rho > 0$ such that $\partial \Tilde{f}_i(B(\overline{x},\rho))$ is bounded.
    As there exists $\ell_0\in \mathbb{N}$ such that $\{x_\ell\}_{\ell\ge \ell_0}
    \subset B(\overline{x},\rho)$, $(b_i^\ell)_{\ell\in \mathbb{N}}$ is a bounded sequence, hence it has a sequential cluster point \cite[Lemma 2.45]{bauschke2017convex}.\\
    Let $(b_i^{\ell_k})_{k\in \mathbb{N}}$
    be a subsequence of $(b_i^\ell)_{\ell\in \mathbb{N}}$ converging to $\overline{b}_i$.
    It follows from \cite[Proposition~16.36]{bauschke2017convex}
    that $\overline{b}_i \in \partial \Tilde{f}_i(\overline{x})$.
\item Assume that $\Tilde{f}_i$ is differentiable at $\overline{x}$. 
Then $\partial \Tilde{f}_i(\overline{x}) = \{\nabla  \Tilde{f}_i(\overline{x})\}$.
Since $(b_i^{\ell})_{\ell\in \mathbb{N}}$ is bounded (see proof of (i)) and has a single sequential cluster point, it converges to $\nabla \Tilde{f}_i(\overline{x})$.
\end{enumerate}
\end{proof}

The following remark will be useful in the proof of convergence of Algorithm \ref{alg: alg1}.
\begin{remark}\label{a:constcurvi}
Let $i \in I$ and let $a_i\in[\rho_{i}/2,+\infty[$. At every iteration $k\in \mathbb{N}$ of Algorithm \ref{alg: alg1} such that
$I(x_k)\neq \varnothing$, the subgradients $(\varphi_i^k)_{i\in I(x_k)}$
belong to the set
\begin{equation}
\{\varphi\colon \mathbb{R}^n\rightarrow\mathbb{R}\mid
(\forall x \in \mathbb{R}^n)\,\varphi(x)=-a_i\|x\|^{2}+\langle b,x\rangle, b\in\mathbb{R}^{n}\}.
\end{equation}
\end{remark}
This means that the same curvature constant $a_i$ remains unchanged through the whole iteration process,
for each constraint function $f_i$, $i\in I$, when computing its abstract $\Phi_{lsc}^{i,0}-$subgradient.

The theorem below, which constitutes our main result, studies the convergence properties of Cutting spheres algorithm \ref{alg: alg1}.
\begin{theorem}
\label{th: convergence}
Consider Problem \eqref{problem}.
Suppose that Assumption \ref{a:compact} holds.
Let $(x_k)_{k\in\mathbb{N}}$ be the sequence generated by  Cutting spheres algorithm \ref{alg: alg1}.
Then,
\begin{enumerate}
    \item\label{th: convergencei}  $(x_k)_{k\in\mathbb{N}}$  possesses a convergent
subsequence, and every sequential cluster point of $(x_k)_{k\in\mathbb{N}}$ is a global solution to \eqref{problem};
\item if  \eqref{problem} has a unique solution, then $(x_k)_{k\in\mathbb{N}}$
converges to this solution;
\item $\big(J(x_k)\big)_{k\in \mathbb{N}}$ is a nondecreasing sequence converging
to the optimal value $J^*$ of \eqref{problem}. 
\end{enumerate}

\end{theorem}
\begin{proof}\ 
\begin{enumerate}
    \item  According to Lemma \ref{le:welldef},
$\{x_k\}_{k\in\mathbb{N}}\subset E$. Moreover, $E$ is a nonempty closed ball in $\mathbb{R}^n$, hence it is compact. We deduce that $(x_k)_{k\in\mathbb{N}}$  possesses a subsequence converging to a point in $E$ and 
that every sequential cluster point of $(x_k)_{k\in\mathbb{N}}$  belongs to  $E$.

Now, let us show that every converging subsequence of $(x_k)_{k\in\mathbb{N}}$
is a minimizing sequence, i.e., its limit $\overline{x}$ is such that ${J(\overline{x})=J^*}$.
We first note that, if at some iteration $k$,  the obtained iterate $x_{k}$ is feasible then, by Lemma \ref{lem: alg1}, $x_{k}=\overline{x}$ is a global solution to Problem \eqref{problem}, which means that the algorithm solves the problem in a finite number of steps. We can thus focus on the case when, for every $k\in \mathbb{N}$,
$x_k$ is infeasible for \eqref{problem}.

Consider a convergent subsequence $(x_{k_\ell})_{{k_\ell}\in\mathbb{N}}$ with limit point $\overline{x}\in E$. Let $j\in I$. 
Two situations may arise:
\begin{enumerate}[start=1,label=(\alph*)] 
\item There exists $\ell_0 \in \mathbb{N}$ such that, for every 
$\ell \ge \ell_0$, the $j$-th constraint is satisfied by $x_{k_{\ell}}$, i.e., 
$f_{j}(x_{k_{\ell}})\le 0$. Then, by continuity of $f_{j}$ on $E$, $f_{j}(\overline{x})\le 0$. 
\item\label{c:infiniteviolation} There exists a subsequence
$(x_{p_\ell})_{\ell \in \mathbb{N}}$ of $(x_{k_\ell})_{{k_\ell}\in\mathbb{N}}$
for which the $j$-th constraint is violated, i.e., for every $\ell \in \mathbb{N}$, $f_{j}(x_{p_\ell})>0$.
\end{enumerate}

Let us investigate Case \ref{c:infiniteviolation}.
By construction, for every $\ell \in \mathbb{N}$ and $p>p_\ell$, ${L_{j}(\varphi^{p_\ell}_{j},x_{p_\ell})(x_{p})\le 0}$
where 
$\varphi^{p_\ell}_{j} \in \partial_{\rm lsc} f_{j}(x_{p_\ell})$.
 
 Since $L_{j}(\varphi^{p_\ell}_{j},x_{p_\ell})(\cdot)$ is continuous,
 we deduce that
 \begin{equation}
     \label{eq: barx}
     L_{j}(\varphi^{p_\ell}_{j},x_{p_\ell})(\overline{x})\le 0.
 \end{equation}
According to Remark~\ref{a:constcurvi}, $\varphi^{p_\ell}_{j}$ has the following 
form:
$$
(\forall x \in \mathbb{R}^n)\quad 
\varphi^{p_\ell}_{j}(x) =-a_j\|x\|^2+\langle b_{j}^{p_\ell},x \rangle,
$$
where $b_{j}^{p_\ell} \in\partial \Tilde{f}_j(x_{p_\ell})$ with 
$\Tilde{f}_j = f_{j}+a_j\|\cdot\|^2$. Since $x_{p_\ell} \to \overline{x} \in E \subset \text{int dom }\Tilde{f}_j$ (see \cite[Corollary 8.39]{bauschke2017convex}), it follows from Theorem  \ref{th: subdifferential convergence}\eqref{th: subdifferential convergencei}
that there exists a subsequence $(b_{j}^{q_\ell})_{\ell \in \mathbb{N}}$ of $(b_{j}^{p_\ell})_{\ell \in \mathbb{N}}$
converging to ${\overline{b}_{j}\in\partial 
\Tilde{f}_j(\overline{x})}$.
Using \eqref{eq: barx} and Lemma~\ref{lem: quad form constr} yields
\begin{equation}\label{e:Ljphiqlxql}
    L_j(\varphi_{j}^{q_\ell}, x_{q_\ell})(\overline{x}) =f_{j}(x_{q_\ell})+\varphi_{j}^{q_\ell}(\overline{x})-\varphi_{j}^{q_\ell}(x_{q_\ell})\le 0.
\end{equation}
Let us look at the behavior of this terms as
$\ell\rightarrow \infty$. By continuity, 
\begin{equation}\label{e:fjqltofibx}
f_j(x_{q_\ell})\rightarrow f_j(\overline{x})\ge0.
\end{equation}
Let us show that $\varphi_{j}^{q_\ell}(\overline{x})-
\varphi_{j}^{q_\ell}(x_{q_\ell})\rightarrow 0$. 
By defining
$$
(\forall x \in \mathbb{R}^n)\quad 
\bar{\varphi}_{j}(x) =-a_j\|x\|^2+\langle \overline{b}_{j},x \rangle,
$$
we obtain the following inequality
\begin{equation}
    \label{eq: 4 modules}
    \begin{aligned}
    |\varphi_{j}^{q_\ell}(\overline{x})-\varphi_{j}^{q_\ell}(x_{q_\ell})|
    &\le |\varphi_{j}^{q_\ell}(\overline{x})-\bar{\varphi_{j}}(\overline{x})|
    +|\bar{\varphi}_j(\overline{x})-\bar{\varphi_{j}}(x_{q_\ell})|+
    |\bar{\varphi_{j}}(x_{q_\ell})-\varphi_{j}^{q_\ell}(x_{q_\ell})|\\
    &=|\langle b_{j}^{q_\ell}-\overline{b}_{j},\overline{x}\rangle|+
    |\bar{\varphi}_j(\overline{x})-\bar{\varphi_{j}}(x_{q_\ell})|+
    |\langle\overline{b}_j-b_{j}^{q_\ell},x_{q_\ell}\rangle|.
    \end{aligned}
\end{equation}
The first and last terms converge to $0$, because of the convergence of $(b_{j}^{q_\ell})_{\ell \in \mathbb{N}}$ to $\overline{b}_j$ and the boundedness
of $(x_{q_\ell})_{\ell \in \mathbb{N}}$. The second one converges to
$0$ by continuity of $\bar{\varphi}_j$. In view of these facts,
we deduce from \eqref{e:Ljphiqlxql} and \eqref{e:fjqltofibx} that $f_j(\overline{x}) = 0$,
which proves that $\overline{x}$ satisfies the $j$-th constraint.\\
By repeating the same reasoning for each constraint $f_j$, $j\in I$ we prove that  $\overline{x}\in \A$, hence $J(\overline{x})\ge J^*$. 
According to Lemma \ref{lem: out set}\eqref{lem: out setii}
and the continuity of $J$,
$$
J(x_{k_\ell})\le J^* \Rightarrow J(\overline{x})\le J^*.
$$
Then, $J^*=J(\overline{x})$ which completes the proof of \eqref{th: convergencei}.
\item if  \eqref{problem} has a unique solution, then $(x_k)_{k\in \mathbb{N}}$
is a bounded sequence with a unique sequential cluster point. Thus, it converges to this point.
\item  By \eqref{eq: inclusions}, $\big(J(x_k)\big)_{k\in \mathbb{N}}$ is a nondecreasing sequence which is upper bounded by $J^*$. In addition,
it follows from \eqref{th: convergencei} that $\big(J(x_k)\big)_{k\in \mathbb{N}}$ has
a subsequence converging to $J^*$. The limit of $\big(J(x_k)\big)_{k\in \mathbb{N}}$ is thus equal to $J^*$.
\end{enumerate}
\end{proof}

\section{Warm Restart}
\subsection{An algorithm with warm restarts}
The number of quadratic constraints of the set  $\mathcal{O}_k$ in Algorithm \ref{alg: alg1} becomes increasingly large as the iteration number $k$ grows. To overcome this drawback,  we introduce a warm restart procedure in the proposed outer approximation approach as detailed in Algorithm \ref{alg: warm_rest}. 

In contrast to Algorithm \ref{alg: alg1}, Algorithm \ref{alg: warm_rest}
can perform two types of iterations, namely \textit{restart iterations} and \textit{cumulative iterations}. At iteration number $k\in \mathbb{N}$, the following notation will be used:
\begin{itemize}
    \item $r_k$ is the index of the last \textit{restart iteration},
    \item $\overline{m}> 0$ is a positive real number, which is used as an upper  
    bound on the computational cost $m_k$ for 
    computing a cumulative iteration at $k$ (typically, $m_k$ is the number of quadratic constraints that $\OAP_k$ would involve if $k$ is a cumulative iteration).
 \end{itemize}
 A \textit{cumulative iterations} keep all the constraints, which were built in the algorithm in the previous steps since the last \textit{restart iteration}. $\mathcal{O}_k$ can still be built recursively as in
 \eqref{e:recurOk}. The difference, however, with the Cutting sphere algorithm \ref{alg: alg1} is
 that we only aggregate constraints of the form
 $L_i(\varphi_i^\ell, x_\ell)(x) \le 0$ for $i\in I(x_\ell)$
 and $\ell \in \{r_k,\ldots,k\}$.

A \textit{restart iteration} is performed whenever the following two conditions are simultaneously met:
\begin{itemize}
    \item the cost at the current iteration
    strictly increases with respect to the last \textit{restart iteration}, i.e., $J(x_k)>J(x_{r_k})$,
    \item the computing cost $m_k$ of a cumulative iteration step at $k$ exceeds the chosen bound $\overline{m}$.
\end{itemize}
 Unlike \textit{cumulative iterations}, a \textit{restart iteration} takes into account only the constraints built at the current iteration and the lower level bound $J(x_k)$, i.e., the approximation set is given by
\begin{equation}\label{iteration:restart}
      \mathcal{O}_k = \{ x \in \mathbb{R}^n \mid 
    (\forall i \in I(x_k))\;
    L_i(\varphi_i^k, x_k)(x) \le 0\;\text{and}\;J(x)\ge J(x_k)\}.
\end{equation} 
A restart iteration is always followed by a cumulative one since the condition
$J(x_{k+1}) = J(x_{r_{k+1}})$ is then satisfied in Algorithm \ref{alg: warm_rest}.
It is worth observing that, because of the quadratic form of function $J$, minimizing $J$ over $\mathcal{O}_k$
remains
a S-QCQP problem.

\begin{algorithm}[H]
\caption{Cutting spheres with warm restarts}
\label{alg: warm_rest}
\begin{algorithmic}[1] 
\ENSURE $x_0 = z$, 
$r_0 = 0$, 
$k = 0$, $\mathcal{O}_{-1}=\mathbb{R}^n$,
$m_{0} = 0$, $s_{-1} = 0$, $\delta > 0$, $\overline{m}>0$
\WHILE {$s_{k-1} = 0$}
\STATE $s_k = 1$ \qquad\% binary stopping variable
\IF{$(\exists\,i\in I)$ $f_i(x_{k}) > 0$}
\IF{$m_{k}\ge\overline{m}$}
\IF {$J(x_k)\textcolor{black}{> } J(x_{r_k})+\delta$} \label{alg:line:ifrestart}
\STATE $r_{k+1}=k+1$\label{alg:line:updater}
\STATE Build $\mathcal{O}_k$  as in \eqref{iteration:restart} \qquad\% $k+1$ is restart iteration
\STATE $s_k = 0$
\ENDIF
\ELSE
\STATE $r_{k+1}=r_k$
\STATE Build $\mathcal{O}_k$ as in  \eqref{e:recurOk} \quad\qquad\% $k+1$ is cumulative iteration.
\STATE $s_k=0$ 
\ENDIF
\IF {$s_k = 0$}
\STATE $x_{k+1} \in  \arg\min
_{x\in \mathcal{O}_k} \; \|x-z\|^{2}$ \qquad\% solution of outer approximation problem \label{alg: warm_rest_line_argmin}
\ENDIF
\ENDIF
\STATE $k\leftarrow k+1$.
\ENDWHILE
\end{algorithmic}
\end{algorithm}

Algorithm \ref{alg: warm_rest} generates a sequence $(x_k)_{k\in \mathbb{L}}$
where  $\mathbb{L}\subseteq \mathbb{N}$ is the set of running iterations.
Let $\mathbb{K}$ be the set of restart iterations, defined as
\begin{equation}\label{set:K}
\mathbb{K}=
\{k\in \mathbb{L}\setminus\{0\} \mid r_{k} \neq r_{k-1}\}.
\end{equation}
So basically $\mathbb{K}$ is the set of iteration numbers for which
line \ref{alg:line:updater} in Algorithm \ref{alg: warm_rest} is performed.
We are introducing $\mathbb{L}\subset \mathbb{N}$, instead of just writing  $k\in \mathbb{N}$ because Algorithm \ref{alg: warm_rest} may stop in a finite number of steps (always for $\delta>0$, as shown in the next proposition).

\begin{example}
Suppose that in Algorithm \ref{alg: warm_rest} we have $\mathbb{L}=\{1,\dots,8\}$ and the restart iteration are: 3 and 6, i.e. $\mathbb{K}=\{3,6\}$. Then
$r_0=r_1=r_2=0$ and then at iteration $3$ restart is performed, therefore $r_3=3$, then $r_4=r_5=3$ and  iteration $6$ is a restart iteration, hence $r_6=6$ and $r_7=r_8=6$. 

\end{example}

Let $\mathcal{S}_{\A}$ be solution set of problem \eqref{problem}. We have the following properties of Algorithm \ref{alg: warm_rest}.
\begin{proposition}\label{prop:warmbis} Consider Problem \eqref{problem}.
Suppose that Assumption \ref{a:compact} holds. 
Let $(\mathcal{O}_{k-1},x_k)_{k\in  \mathbb{L}}$ be generated by Algorithm \ref{alg: warm_rest}.
\begin{enumerate}
    \item\label{prop:warmi} 
    For every $k\in \mathbb{L}$, we have
    $\mathcal{O}_{k-1} \subseteq \mathcal{O}_{k-2} \subseteq \ldots \subseteq \mathcal{O}_{r_{k}-1}$.
    \item\label{prop:warmii}  For every $k\in \mathbb{L}$,
    $
    \mathcal{A}\subseteq \mathcal{O}_{k-1}$.
    \item\label{prop:warm3} $(J(x_k))_{k\in \mathbb{L}}$ is a nondecreasing sequence bounded from above
    by $J^*$ and $(J(x_k))_{k\in \mathbb{K}}$ is an increasing sequence.
    \item\label{prop:warmiii}  For every $k\in \mathbb{L}$, $I(x_k) = \varnothing$ if and only if $x_k \in \mathcal{S}_{\mathcal{A}}$.
    \item\label{prop:warm4} If $\delta > 0$, then
    \begin{enumerate}
    \item for every $k\in \mathbb{L}$, the number of constraints in  $\mathcal{O}_{k-1}$ never exceeds $\overline{m}$;
    \item there exists ${k^*}\in \mathbb{N}$ such that $s_{{k^*}} = 1$ (so, the algorithm "stops in a finite number of steps");
    \item either 
    $x_{{k^*}} \in \mathcal{S}_{\mathcal{A}}$, or 
    the resulting lower approximation to $J^*$ is
    $J(x_{{k^*}})$.
    \end{enumerate}

\end{enumerate}
\end{proposition}

\begin{proof}\ 
\begin{enumerate} 
    \item These inclusions follow from the fact that, between two consecutive restarts, Algorithm~\ref{alg: warm_rest} accumulates constraints.
    \item Let us show the property by induction. It is obviously satisfied when $k=0$. Assume that $k+1\in \mathbb{L}$ and that the property is satisfied at index $k$ and let $x\in \mathcal{A}$. Two cases may arise. 
    If $\mathcal{O}_k$ is built according to \eqref{e:recurOk}, since 
    $x\in \mathcal{O}_{k-1}$ and 
    \begin{equation}\label{e:xbarOkk}
    x \in \mathcal{A}\subseteq  \{ y \in \mathbb{R}^n \mid 
    (\forall i \in I(x_k))\;
    L_i(\varphi_i^k, x_k)(y) \le 0\},
    \end{equation}
    we deduce that $x \in \mathcal{O}_k$.
    When $\mathcal{O}_k$ is built by \eqref{iteration:restart},
    since $x_k$ minimizes $J$ on $\mathcal{O}_{k-1}$ and $x\in \mathcal{O}_{k-1}$, $J(x) \ge J(x_k)$. Combining this inequality with
    \eqref{e:xbarOkk} and \eqref{iteration:restart} shows that 
    $x\in \mathcal{O}_k$.
    \item     In a cumulative step, the non-decreasing property follows from
    \eqref{prop:warmi}. In turn, the last inequality constraint in  \eqref{iteration:restart} ensures that $J(x_{k+1})\geq J(x_k)$
in a restart iteration.
    The fact that, for every $k\in \mathbb{L}$, $J(x_k)\le J^*$ is a consequence of
    \eqref{prop:warmii}.  Let $(k_1,k_2)\in \mathbb{K}^2$ be such that
    $k_1= \max\{k\in \mathbb{K}\mid k<k_2\}$. Then
    $J(x_{k_2}) >  J(x_{k_1})+\delta$.
    This shows that $(J(x_k))_{k\in \mathbb{K}}$ is an increasing sequence with jumps greater than $\delta$.
    \item\label{prop:warmiiii} $I(x_k) = \varnothing$ $\Leftrightarrow$ $x_k\in \mathcal{A}$.
    Since  $J(x_k)\le J^*$, this is also
    equivalent to $x_k \in \mathcal{S}_{\A}$.
    \item
    \begin{enumerate}
\item This follows from the construction of  Algorithm \ref{alg: warm_rest}. 

\item  The algorithm stops at iteration $k+1$ if $s_k =1$,
since we exit from the while loop.

Two situation may appear when running Algorithm \ref{alg: warm_rest}. 
\begin{enumerate}[start=1,label={\arabic*}.]
    \item\label{case:1_globmin1} 
    There exist $k^*\in \mathbb{L}$ such that $I(x_{k^*})=\varnothing$.
    \item\label{case:2_globmin1} For any $k\in \mathbb{L}$, $I(x_{k})\neq \varnothing$.
\end{enumerate}
In case \ref{case:1_globmin1}, by (iv), $x_{k^*}\in \mathcal{S}_{\A}$ and, when iteration $k^*+1$ starts, $s_{k^*}=1$.
Let us show that, in case \ref{case:2_globmin1}, the algorithm also terminates in a finite number of steps.

We will show that the number of restart iterations is finite.

By (iii), for every $(k_1,k_2) \in \mathbb{K}^2$ with $k_2>k_1$,
$$ J(x_{k_2})>J(x_{k_1})+\delta$$
We deduce that, for every $n\in \mathbb{N}$,  $(k_1,k_2,\dots,k_n) \in \mathbb{K}^n$ with $k_n>k_{n-1}>\dots>k_1$, we have
$$J(x_{k_n})\ge J(x_{k_1})+(n-1)\delta
.$$

However, $J(x_{k_n})\leq J^*<+\infty$. Therefore, the set $\mathbb{K}$ must be finite and 
the following holds:
\begin{equation}
    \label{finite_restart}
    (\exists\,\overline{k}\in \mathbb{K})
    (\forall k \in \mathbb{L} \cap [\overline{k},+\infty[)\quad r_k= \overline{k}.
\end{equation}

This means that, for every $k\in \mathbb{L} \cap [\overline{k},+\infty[$ the set $\mathcal{O}_{k}$ can be built only as in  \eqref{e:recurOk} and the number of constraints $m_k$ will be increasing for $k\geq \bar{k}$.
Hence, there exists $k^*\ge \bar{k}$ such that $m_{k^*}\ge \overline{m}$. As no restart occurs,
$s_{k^*}=1$ and 
Algorithm \ref{alg: warm_rest} stops at iteration $k^*+1$. 
\item 
The proof follows from \eqref{prop:warm3} and \eqref{prop:warmiii}.

\end{enumerate}
    \end{enumerate}
    \end{proof}

    \begin{remark}
   
    The same conclusion as in Proposition \ref{prop:warmbis}\eqref{prop:warm4} holds
    if $\delta$ in Algorithm \ref{alg: warm_rest} is replaced by
    $\delta_k = \overline{\delta}(r_k+\epsilon)^{-1}$ with 
    $\overline{\delta}> 0$ and
    $\epsilon > 0$.
    \end{remark}

{\color{black}
\subsection{Outcome of Algorithm \ref{alg: warm_rest} Cutting spheres with warm restarts}

According to Proposition~\ref{prop:warmbis}, for any $\delta>0$, Algorithm \ref{alg: warm_rest} 
stops in a finite number of $k$ steps with either (a) $x_{k}$ being feasible or (b) $x_{k}$ being infeasible. Case (a) implies that the original problem \eqref{problem} is solvable in a finite number of steps (by arguments analogous to those in Lemma \ref{lem: alg1}).

On the other hand, in case (b), a natural and interesting question is how far  from the feasibility or equivalently, how far from the optimality  is the outcome $x_{k}$. An answer to this question can be provided by investigating the so-called {\em global error bound} property as explained below.
 
 Recall that the feasible set $\A\subset\mathbb{R}^{n}$ of Problem \eqref{problem} has the form \eqref{e:defA}.

We say that $\A$ enjoys the {\em global error bound property} if there exists a constant $\tau>0$ such that
\begin{equation} 
\label{eq:error_bound} 
(\forall z \in \mathbb{R}^{n})\quad 
d(z,\A)\le \tau\max_{1\le i\le m}(f_{i}(z))_{+}\ ,
\end{equation}
where $d(z,\A)=\inf_{a\in\A}\|z-a\|$ and, for every $i\in \{1,\ldots,m\}$, $(f_{i}(z))_{+}:=\max\{0,f_{i}(z)\}$.
 A lot of research has been devoted to necessary/sufficient conditions ensuring \eqref{eq:error_bound}  in the case when $m=1$ and $f_1$ is convex, see e.g. \cite{cuong2022error, fabian2010error}. 
 For $m> 1$ and $(f_{i})_{1\le i \le m}$ convex,  sufficient conditions can be found in \cite{lewis1998error}.
 
 It is not our aim here to provide a detailed analysis of the error bound inequality \eqref{eq:error_bound} when, for every $i\in I$, $f_{i}$ is $\rho_i$-weakly convex. 
 Instead, we observe that sufficient (and necessary) conditions ensuring the error bound inequality for  convex functions $(f_{i})_{1\le i \le m}$ almost straightforwardly lead to conditions ensuring the error bound property for weakly convex functions.
 
 For a given $i\in I$, $f_i$ is proper, lsc, and $\rho_i$-weakly convex with $\rho_i \in \mathbb{R}_+$. 
 Then, the function $\tilde{f}_{i}$ defined as 
 $$
 \tilde{f}_{i}:=f_{i}+\frac{\rho_{i}}{2}\|\cdot\|^{2},\ 
 $$
 is proper, lsc, and convex.

Let
$$
\widetilde{\A}=\{x\in \mathbb{R}^n\mid (\forall i \in I)\, \tilde{f}_{i}(x)\le 0\}.
$$ 
Clearly,
\begin{equation} 
\label{eq:inclusion1} 
\widetilde{A}\subset\A.
    \end{equation}
Hence, by \cite[Proposition 8, pp. 105-106]{lewis1998error}, we obtain the following fact.

\begin{proposition}
\label{prop_8_Lewis_pang} 
Assume that
$$ 
\operatorname{int}\widetilde{A}=\{x\in \mathbb{R}^n \ |\ (\forall i \in I)\;\tilde{f}_i(x)<0\}\neq\varnothing 
$$ and there exists a scalar $\kappa>0$ such that
\begin{equation} 
\label{eq:Suff_error_bound} 
(\forall\ \bar{x}\in \operatorname{bd}\widetilde{A})\ \ (\exists\ \bar{a}\in \operatorname{int}\widetilde{A})\quad \frac{\|\bar{x}-\bar{a}\|}{\min_{1\le i\le m} -\tilde{f}_i(\bar{a})}\le\kappa.
\end{equation}
Then there exists a constant $\tau>0$ such that
\begin{equation} 
\label{eq:Suff_error_bound1}
(\forall x \in \mathbb{R}^n)\quad
\operatorname{dist}(x,\A)\le \operatorname{dist}(x,\widetilde{A})\le\tau\max_{1\le i\le m}(\tilde{f}_i(x))_{+}.
\end{equation}
\end{proposition}

As a consequence of this result, if 
$\text{int\,}\widetilde{A}\neq\varnothing$,
 \eqref{eq:Suff_error_bound}  holds and,for every $i\in I(x_{k})$,
$f_{i}(x_{k})>0$  then, by \eqref{eq:Suff_error_bound1},
$$
\operatorname{dist}(x_{k},\A)\le \tau\max_{i\in I(x_{k})}(f_{i}(x_{k})+\frac{\rho_{i}}{2}\|x_{k}\|^{2}),
$$
}

\section{Hidden convexity of
the quadratic outer approximation problem}\label{sec:hidden_conv}

In this section we investigate more closely the problem of  outer approximation. At any iteration $k\in \mathbb{N}$ of our proposed algorithms, the outer approximation problem for \eqref{problem}, introduced in Definition \ref{def: out approx}, is of the form:
\begin{equation}
    \label{outerproblem}
    \tag{OP$_k$}
    \minimize{x\in \OAP_k} \|x-z\|^2,
\end{equation}
where, 
\begin{equation} 
\label{oaaset0}
\OAP_k =\{x\in \mathbb{R}^n\mid (\forall i \in \{1,\ldots,m_k\})\; q_i^k(x) \le 0\}
\end{equation}
with, for every $i\in \{1,\ldots,m_k\}$ and $x\in \mathbb{R}^n$, $q_i^k(x) := -a_{i}\|x\|^2+\langle b_i^k,x\rangle+c_i^k$.
Problem \eqref{outerproblem} belongs to the class of S-QCQP problems studied in \cite{bednarczuk2023global}, where the objective function has the following form:
\begin{equation}
\label{sqcqp_eq}
(\forall x \in \mathbb{R}^n)\quad q_0(x):= a_0\|x\|^2+\langle b_0, x\rangle +c_0
\end{equation}

with $a_0=1$, $b_0 = 2 z$, and $c_0=\|z\|^2$.

As already noticed, there always exists an optimal solution $x^*$ for \eqref{outerproblem}, since $q_0$ is coercive, and $\OAP_k$ is nonempty and closed.

\subsection{Hidden convexity under dimension condition}

In this subsection we recall that KKT conditions for \eqref{outerproblem} are necessary and sufficient for optimality under the following condition.

\begin{condition}(Dimension condition)\label{cond:number_constr}
The problem \eqref{outerproblem} satisfies the so-called dimension condition if any maximally linearly independent family of vectors extracted from $\{b_i^k \mid i\in \{0,\ldots,m_k\}\}$ has cardinality strictly less than $n$. 
\end{condition}

By \cite[Theorem 4.2]{bednarczuk2023global}, the following result holds.

 \begin{theorem}
\label{sqcqpcharcter}
Consider problem \eqref{outerproblem}. 
Let $x^*$ be 
a feasible point.\\
If there exists a point $\widetilde{x}\in \mathbb{R}^n$ such that 
\begin{equation*} 
(\forall i \in \{1,\ldots,m_k\})\quad 
q_i(\widetilde{x})<0.
\end{equation*}
Then, the conditions

\begin{equation}
    \label{kkt2}
    \tag{KKT}
    \big(\exists (\gamma_1,\ldots ,\gamma_{m_k})\in\mathbb{R}^{m_k}_+\setminus \{0\}\big)\quad
    \begin{cases}
        &(i)\ \ \
        2a_0x^*+b_0+\sum\limits_{k=1}^{m_k} \gamma_i (b_i-2a_ix^*)=0,\\
        &(ii)\ \ \gamma_i q_i(x^*)=0,\ \ i\in \{1,\ldots,m_k\},\\
        &(iii)\ \ \ a_0-\sum\limits_{i=1}^{m_k}\gamma_ia_i\ge0,
    \end{cases}
\end{equation}
are sufficient for global optimality of $x^{*}$.
Moreover, suppose that Condition \ref{cond:number_constr} holds. 
Then, conditions \eqref{kkt2} 
are necessary and sufficient for global optimality of $x^{*}$.
\end{theorem}

 The dual of the Lagrangian Dual of a general QCQP problem is a SDP (Semi-Definite Program) relaxation of the QCQP, which is a convex program, see \cite{Vanden1}.
When the solution set of QCQP is included in the solution set of this SDP relaxation, we refer to this situation as hidden convexity, see \cite{ben2011hidden,ben1996hidden}.

More precisely, the classical Lagrangian dual of \eqref{outerproblem} is defined as follows:

\begin{align}\label{problem: maxmin_lagr}\tag{LD}
    &\max_{\gamma = (\gamma_i)_{1\le i \le m_k}\geq 0} \min_{x\in \mathbb{R}^n} \left(L(x,\gamma)
    :=\|x-z\|^2 + \sum_{i=1}^{m_k} \gamma_i q_i(x)\right),
\end{align}
which can be rewritten as the following semidefinite program (see, e.g.,   \cite[Section 3]{Vanden1} and \cite[Theorem 1]{bao2011semidefinite}):
\begin{align}\label{lagrangian_dual_eq}
   & \max_{t,\gamma\geq 0} t,\quad \text{s.t.}\ 
      \left[\begin{array}{cc}
        I & -z \\
        -z^\top & \|z\|^2-t  
    \end{array}\right]
     + \sum_{i=1}^{m_k} \gamma_i \left[\begin{array}{cc}
          -a_i I  & b_i  \\
          b_i^\top & c_i  
     \end{array}\right] \succeq 0,
\end{align}
where, for every symmetric matrix $A\in S^n$, the notation $A\succeq0$ means that matrix $A$ is positive semidefinite.

Problem \eqref{lagrangian_dual_eq} is a cone programming problem, and its Lagrangian dual, for a suitably defined Lagrangian (see \cite{shap1}), takes the Shor form:
\begin{equation}\label{problem: sdp}\tag{SDP}
    \min_{X \in S^n, x\in \mathbb{R}^n} \operatorname{tr} (X) - 2 z^\top x{+\|z\|^2},\quad \mbox{s.t.} \
    \begin{cases}
    -a_i\operatorname{tr}(X) + b_i^\top x +c_i \leq 0,\ i\in \{1,\dots,m_k\}\\ X-xx^\top \succeq 0,
    \end{cases}
 \end{equation}

where $\operatorname{tr}$ is the trace operator.

According to \cite[Proposition 2]{bao2011semidefinite} (see also \cite{kim2023equivalent}),  $\operatorname{Opt}\eqref{problem: maxmin_lagr}$ and $\operatorname{Opt}\eqref{problem: sdp}$ are equal if there exists $(X,x)\in S^n\times \mathbb{R}^n$ such that 
\begin{equation}
    \label{slater_SDP_feas}
    (\forall i\in \{1,\ldots,m_k\})\quad 
    {-}a_i\operatorname{tr}(X) + b_i^\top x +c_i < 0.
\end{equation}
Here, $\operatorname{Opt}(\cdot)$ denotes the optimal value of the optimization problem.
Every element $\widetilde{x}$ in the quadratic constraint sets defining \eqref{outerproblem} is such that
$$(\forall i \in \{1,\ldots,m_k\})\quad 
q_i(\widetilde{x})={-}a_i\|\widetilde{x}\|^2 + b_i^\top \widetilde{x} +c_i \le 0.$$
Since the complement of the constraint set of \eqref{outerproblem} is the union of a finite number of open balls in $\mathbb{R}^n$, there exists a point $\widetilde{x}\in \mathbb{R}^n$ such that 
$$(\forall i \in \{1,\ldots,m_k\})\quad {-}a_i\|\widetilde{x}\|^2 + b_i^\top \widetilde{x} +c_i < 0.$$
Then, by setting $X=\widetilde{x}\widetilde{x}^\top$ and $x=\widetilde{x}$,
\eqref{slater_SDP_feas} is satisfied and $\operatorname{Opt}\eqref{problem: maxmin_lagr}=\operatorname{Opt}\eqref{problem: sdp}$.

 In addition, if Condition \ref{cond:number_constr} holds, it follows from
  Theorem \ref{sqcqpcharcter} that the KKT conditions define a saddle point of the Lagrange function. Then strong duality holds and 
  \begin{equation}\label{eq:optimal_equiv}
    \operatorname{Opt}\eqref{problem: sdp}=\operatorname{Opt}(7.3)=\operatorname{Opt}\eqref{outerproblem}.
\end{equation}
Note that every feasible point $x$ for \eqref{outerproblem} is such that
$(X,x) = (xx^\top,x)$ is a feasible point for \eqref{problem: sdp}.
If $x^*$ is a solution to \eqref{outerproblem} , then
\[
\operatorname{Opt}\eqref{outerproblem}=\|x^*-z\|^2  = \operatorname{tr}(X^*)-2z^\top x^*
+\|z\|^2,
\]
where $X^* = x^*(x^*)^\top$. Thus, it follows from \eqref{eq:optimal_equiv}
that $(X^*,x^*)$ is a solution to \eqref{problem: sdp}.
This means that under Condition  \ref{cond:number_constr}, \eqref{outerproblem} is a hidden convex problem.

{\color{black}
\subsection{The feasibility problem of finding a point on sphere and in a polyhedron}
\label{sec: feas quad }
As pointed out in the previous subsection, when Condition \ref{cond:number_constr} holds, the optimal value of problem \eqref{outerproblem} can be found by solving a SDP relaxation.
On the other hand, after finding the optimal value  $\operatorname{Opt}\eqref{outerproblem}$, we still need to recover an optimal solution for \eqref{outerproblem}, since we recall that, when $\operatorname{Opt}\eqref{outerproblem}=\operatorname{Opt}\eqref{problem: sdp}$, the set of solutions to \eqref{outerproblem} is a subset
of the set of solutions to \eqref{problem: sdp}.

Hence, in this subsection, for a given $k\in \mathbb{N}$, we consider the resolution of \eqref{outerproblem} when $\operatorname{Opt}\eqref{outerproblem}$ is known, say equal to some $\alpha_k$. 

Without loss of generality, by shifting the variable, we can assume that $J\colon x \mapsto \|x\|^2$, i.e. $z=0$.
Note that adding the constraint $q_{\alpha_k}(x):=\|x\|^2-\alpha_k= 0$
is equivalent to add the constraint
$q_{\alpha_k}(x)\ge 0$, which shows that it does not affect the 
S-QCQP structure of  Problem~\eqref{outerproblem} and Condition \ref{cond:number_constr}.

Then, we notice that every quadratic constraint of the form $q_i(x)=-a_i\|x\|^2+b_{i,k}^\top x+c_k\le0$, $i\in \{1,\ldots,m_k\}$, of \eqref{outerproblem}, generated by Algorithm \ref{alg: alg1} or  \ref{alg: warm_rest}, intersects the sphere $\mathcal{S}_{\alpha_k}:=\{x\in\mathbb{R}^n\ |\ \|x\|^2-\alpha_k=0\}$.

 We say that the $i$-th constraint is redundant at level $\alpha_k$ if the following holds
\begin{equation*}
    \mathcal{S}_{\alpha_k} \cap \{ x \in \mathbb{R}^n
 \mid -a_i\|x\|^2+b_{i,k}^\top x+c_k\le 0\} = \mathcal{S}_{\alpha_k}
 \end{equation*}

 The following result holds.

\begin{lemma}
\label{lem: from out to feas}

Let $k\in \mathbb{N}$ and let $\alpha_k = \operatorname{Opt}\eqref{outerproblem}$.
Consider the
outer approximation problem \eqref{outerproblem},

with
$$
\mathcal{O}_k:=\{x\in X\mid  
L_i(\varphi_i^k
,x_{k})(x) = -a_i \|x\|^2+b_{i,k}^\top x + c_{i,k}\le 0,\ \ i\in I(x_k)\}\cap\bigcap\limits_{p=r_k}^{k-1}\{\mathcal{O}_p\},
$$
where $r_k$ is the first restart iteration before $k$ (set $r_k=0$ when considering Cutting spheres algorithm \ref{alg: alg1}) and 
the redundant constraints at level $\alpha_k$ can be removed.
This problem is equivalent to  the following auxiliary feasibility problem between a sphere and a convex set:
\begin{equation}
    \label{Feasibility_prob}
    \text{ Find }x\in \mathcal{S}_{\alpha_k}\cap \mathcal{O}_k^{\alpha_k},
\end{equation}
where 
\begin{align*}
    &{S}_{\alpha_k}:=\{x\in\mathbb{R}^n\ |\ \|x\|^2=\alpha_k\}\\
    & \mathcal{O}_k^{\alpha_k}:=\{x\in \mathbb{R}^n\mid  b_{i,k}^\top x\le a_i\alpha_k+c_{i,k},\ \ i\in I(x_k)\}\cap\bigcap\limits_{t=r_k}^{k-1}\{\mathcal{O}_t^{\alpha_k}\}.
\end{align*}
\end{lemma}

\begin{proof}
Since $\operatorname{Opt}\eqref{outerproblem}=\alpha_k$, $x$ is an optimal solution to Problem \eqref{outerproblem} if and only if it satisfies $\alpha_k=\|x\|^2$ and $x \in \mathcal{O}_k$. Since, for every $i\in I(x_k)$,
$$
 L_i(\varphi_i^k
,x_{k})(x)=b_{i,{k}}^\top x +a_i\|x\|^2+c_{i,k}=b_{i,{k+1}}^\top x+a_i\alpha_k+c_{i,k}.
$$
this is also equivalent to $x\in \mathcal{S}_k\cap \mathcal{O}_k^{\alpha_k}$.
\end{proof}

\begin{remark}
\label{rem: out feas}
Conversely to Lemma \ref{lem: from out to feas}, if Problem \eqref{Feasibility_prob} has no solution for some $\alpha_k\ge0$, it is not equivalent to \eqref{outerproblem}, which always has a global solution, as already mentioned st the beginning of these section.
Then, by contraposition of Lemma \ref{lem: from out to feas}, it holds $\alpha_k \neq \operatorname{Opt}\eqref{outerproblem}$.
\end{remark}

In view of Remark \ref{rem: out feas}, solving \eqref{outerproblem} reduces to finding the intersection of a sphere and the polyhedron $\OAP_k^{\alpha_k}$.
The Inexact cutting sphere algorithm \ref{alg: inexact} formulated in the next section, takes advantage of Lemma \ref{lem: from out to feas} and Remark \ref{rem: out feas} for solving  \eqref{outerproblem}.

Problem \ref{Feasibility_prob} can be solved, for example by finding a point $\hat{x}_1\in \OAP_k^{\alpha_k}\cap \{x\in \mathbb{R}^n\ |\ \|x\|^2\le \alpha_k\}$  and a point $\hat{x}_2\in\OAP_k^{\alpha_k}\cap \{x\in \mathbb{R}^n\ |\ \|x\|^2\ge \alpha_k\}$, provided that they exist.
Then, since $\OAP_k^{\alpha_k}$ is a convex set, the segment connecting $\hat{x}_1$ and $\hat{x}_2$ will intersect the sphere $\mathcal{S}_{\alpha_k}$ in a point $\hat{x}_3\in \OAP_k^{\alpha_k}$ (hence $x_3$ is a solution to Problem \ref{Feasibility_prob}).
The point $\hat{x}_1$ can be found by solving the convex problem $$\minimize{x\in \OAP_k^{\alpha_k}}\ \ \|x\|^2,$$
while $\hat{x}_2$ can be found by solving the following concave optimization problem.
\begin{equation}
    \label{concave prob}
    \maximize{x\in \OAP_k^{\alpha_k}}\ \ \|x\|^2.
\end{equation}

Problem \eqref{concave prob} is NP-Hard, but it is intensively discussed in the literature, see e.g. \cite{chazelle1993optimal,khachiyan2009generating,pardalos1986methods}. (In the case when $\OAP_k^{\alpha_k}$ is unbounded, 
we can stop the optimization procedure at any point $\hat{x}_2\in \OAP_k^{\alpha_k}$ such that
$\|\hat{x}_2\|^2 \ge \alpha_k$.)

}

\section{Inexact cutting sphere algorithm}

In this section, we propose a modification of the Warm restart outer approximation algorithm~\ref{alg: warm_rest}. 
At each iteration $k$, we solve the outer approximation problem \ref{outerproblem} 

by leveraging  Lemma \ref{lem: from out to feas} and Remark \ref{rem: out feas}.

The outer approximation Algorithm \ref{alg: inexact} described in this section finds an $\varepsilon$-solution to Problem \eqref{problem}, for an arbitrary $\varepsilon>0$.
The concept of $\varepsilon$-solution is defined as follows.

\begin{definition}
\label{def:eps_sol}
Let $J^*$ be the optimal value of a general optimization problem 
\begin{equation}
\label{prog: general}
\tag{GP}
    \minimize{x\in \mathcal{C}}J(x) 
\end{equation}
where $\mathcal{C}\subset \X$ is closed and $J:\X\rightarrow (-\infty, +\infty]$ is proper. 
$x^\varepsilon\in \X$ is an $\varepsilon$-solution to Problem~\eqref{problem} with $\varepsilon\ge0$ if and only if
\begin{equation}
    \label{eq:eps_sol}
    x^\varepsilon\in \{x\in \mathcal{C}\ |\ J(x)\le J^*+\varepsilon\}.
\end{equation}
\end{definition}

In our context, the objective function is $J=\|\cdot\|^2$ (i.e., $z$ is assumed to be equal to 0, by simple translation of the sought variable).

We rewrite Problem \eqref{outerproblem} as the feasibility problem \eqref{Feasibility_prob} by choosing $\alpha_k=J(x_k)$.
By Lemma~\ref{lem: from out to feas}, if $\alpha_k = \operatorname{Opt}\eqref{outerproblem}$, a solution to the feasibility problem \eqref{Feasibility_prob} exists which is a global solution to 
Problem \eqref{outerproblem}.
In contrast, by Remark \ref{rem: out feas}, if the feasibility problem \eqref{Feasibility_prob} does not have any solution, the optimal value of Problem \eqref{outerproblem} is $\operatorname{Opt}\eqref{outerproblem}\neq\alpha_k$.
This analysis is refined in Lemma \ref{lem: sub1 outcome} hereafter.
In the latter case, we perform a restart iteration.

In the proposed algorithm, Subroutine \ref{alg: outsolver}
will play a prominent role.

\makeatletter
\renewcommand*{\ALG@name}{Subroutine}
\makeatother
\setcounter{algorithm}{0}

\begin{algorithm}[H]
\caption{Outer problem solver for chosen $\varepsilon>0$.}
\label{alg: outsolver}
\begin{itemize}
    \item Let $\alpha_k=J(x_k)$.
    Let $\mathcal{S}_{\alpha_k}$ be the sphere centered at zero with radius $\alpha_k$ and let the polyhedron $\mathcal{O}_k^{\alpha_k}$ be defined as in  Lemma \ref{lem: from out to feas}.
     \bigskip
\begin{enumerate}[start=1,label={\arabic*:}]

    \item  
    If $\mathcal{S}_{\alpha_k}\cap\mathcal{O}_k^{\alpha_k}\neq \varnothing$, set $x_{k+1}\in \mathcal{S}_{\alpha_k}\cap\mathcal{O}_k^{\alpha_k}$.
  
    \bigskip
    \item  Else take any point ${x_{k+1}\in \{x\in \mathbb{R}^n\ |\ J(x)=\alpha_k+\varepsilon \}}$. 
  
      \bigskip
    \item  \textbf{Return}  $x_{k+1}$.
    \end{enumerate}
\end{itemize}
\end{algorithm}

Let $m_k$ be the cost of performing Subroutine \ref{alg: outsolver}. The cost $m_k$ is set before we run Subroutine \ref{alg: outsolver} (for example $m_k$ can be the number of quadratic constraints in \eqref{outerproblem}).
Similarly to Inexact cutting spheres Algorithm \ref{alg: warm_rest}, we
define an acceptable upper bound  $\bar{m}$ in terms of complexity.

\makeatletter
\renewcommand*{\ALG@name}{Algorithm}
\makeatother
\setcounter{algorithm}{2}

\begin{algorithm}[H]
\caption{Inexact cutting spheres}\label{alg: inexact}
\begin{algorithmic}[1] 
\ENSURE {$x_0 = 0$, $r_0 = 0$, $k = 0$, $\mathcal{O}_{-1}=
\mathbb{R}^n$, $m_{0} = 0$, $\varepsilon>0$.}
\LOOP
\IF{$I(x_k)\not =  \varnothing$ and {$m_{k}\le \bar{m}$}}  \label{cond2_inexact}
\STATE Build $\mathcal{O}_k$ as in  \eqref{e:recurOk}
\label{line4alg3}
\STATE Apply Subroutine \ref{alg: outsolver} with $\varepsilon$ to find $x_{k+1}$
\IF{$x_{k+1}$ was computed with Subroutine \ref{alg: outsolver}.2}
\STATE Set $r_{k+1}=k+1$. 
\STATE Redefine $\mathcal{O}_k= \{ x \in \mathbb{R}^n \mid \|x\| \geq \|x_{k+1}\|\ \}$ \label{inexact: inexact restart} \quad\% $k+1$ is a restart iteration
\label{alg3 step: redefine}
\ELSE
\STATE Set $r_{k+1}=r_k$\quad\% $k+1$ is a cumulative iteration
\ENDIF
\ELSE 
\STATE \textbf{exit loop}
\ENDIF
\STATE $k\leftarrow k+1$.
\ENDLOOP
\end{algorithmic}
\end{algorithm}

\makeatletter
\renewcommand*{\ALG@name}{Subroutine}
\makeatother
\setcounter{algorithm}{1}

Let $\mathbb{K}$ be the set of restart iterations, defined as in \eqref{set:K}.
So $\mathbb{K}$ is the set of iteration numbers for which line \ref{inexact: inexact restart} in Algorithm \ref{alg: inexact} is performed.
The following assumption is needed in order for Inexact cutting spheres Algorithm \ref{alg: inexact} to be well defined.

\begin{assumption}\label{assumption:epsilon}
The parameter $\varepsilon>0$ is such that for any $\varepsilon_1\in (0,\varepsilon)$  $\A\cap  \operatorname{lev}_{=J^*+\varepsilon_1}J\neq\varnothing
$, where $\operatorname{lev}_{=J^*+\varepsilon_1}J:=\{x \in \mathbb{R}^n \ |\ J(x)= J^*+\varepsilon_1 \}$.
\end{assumption}

Similarly to Algorithm \ref{alg: warm_rest}, Algorithm \ref{alg: inexact} generates a sequence $(x_k)_{k\in \mathbb{L}}$
where we denote with  $\mathbb{L}\subset \mathbb{N}$ the set of running iterations.
The following lemma describes the possible outcomes of Subroutine \ref{alg: outsolver}.

\begin{lemma}
\label{lem: sub1 outcome}
Let $k\in \mathbb{L}$, let $\alpha_k=J(x_k)$, and 
let $J_{k+1}$ be 
 the optimal value  of  the outer approximation problem \eqref{outerproblem}.

When Subroutine \ref{alg: outsolver}.1 is performed, 
$x_{k+1}$ is a global solution to \eqref{outerproblem} and $J_{k+1}=\alpha_k$.

When Subroutine \ref{alg: outsolver}.2 is performed, then 
$J_{k+1}> \alpha_k$.
\end{lemma}
\begin{proof}
Notice that $\alpha_k = J(x_{r_{k}})$ and, by definition of $\mathcal{O}_{-1}$ and line \eqref{alg3 step: redefine} in Algorithm \ref{alg: inexact}, the constraint $\{x\in \mathbb{R}^n \mid \|x\|^2\ge \alpha_k\}$ belongs to the outer approximation set $\OAP_k$, which is the feasible set of  \eqref{outerproblem}. Thus,
\begin{equation} \label{e:boundalphakalgo3}
    (\forall x \in \OAP_k)\quad J(x)\ge \alpha_k.
\end{equation}

Assume that Subroutine \ref{alg: outsolver}.1 is performed at iteration $k$.
Then, the new point $x_{k+1}$ is a solution to problem \eqref{Feasibility_prob}. Hence, $x_{k+1}$ is feasible for $\OAP_k$ with $J(x_{k+1})=\alpha_k$.
According to \eqref{e:boundalphakalgo3},
$x_{k+1}$ is a global minimizer of $J$ over $\OAP_k$.

On the other hand, let Subroutine \ref{alg: outsolver}.2 be performed. Then, no point $x\in \OAP_k$ satisfies $J(x)= \alpha_k$. Since $\alpha_k$ is a lower bound for $J$ on $\OAP_k$, we deduce that, for every $x\in \OAP_k$, $J(x)> \alpha_k$. In addition, we know that there 
exists $\widehat{x}_k\in \OAP_k$ such that $J_{k+1}= J(\widehat{x}_k) > \alpha_k$.
\end{proof}

The geometrical interpretation of Lemma \ref{lem: sub1 outcome} is as follows.

\begin{remark}
When the feasibility problem \eqref{Feasibility_prob} has no solution, the polytope $\OAP_k^\alpha$ is inside $B(0,\sqrt{\alpha_k})$.
The constraint set $\OAP_k$ is the complement of the union of open balls $q_i(x)>0$, $i\in \{1,\ldots,m_k\}$ (see Section \ref{lem: geo_sphere}). 
When the feasibility problem \eqref{Feasibility_prob} has no solution,
these open balls completely cover the level set 
$\operatorname{lev}_{=\alpha_k}J$.
Hence, the feasible set $\A$ of \eqref{problem} lies outside $\operatorname{lev}_{=\alpha_k}J$. 
\end{remark}

As a consequence of Lemma \ref{lem: sub1 outcome}, the following result justifies that Algorithm \ref{alg: inexact} builds a sequence of outer approximations to the feasibility set $\A$ of \eqref{problem}.

\begin{corollary}
\label{corollarytolemma}
Let $k \in \mathbb{L}$ be an iteration of Algorithm \ref{alg: inexact}. 
Consider the previous restart iteration $r_k$.
Suppose that Assumption \ref{assumption:epsilon} holds and $r_k$ is not the last restart iteration (i.e., $r_k \neq \max \mathbb{K}$).
Then, the set $\OAP_{r_k-1}$ defined in line \ref{inexact: inexact restart} of Algorithm \ref{alg: inexact} is an outer approximation to the feasibility set $\A$ of \eqref{problem} and $J(x_{r_k}) \le J^*$.
\end{corollary}
\begin{proof}
Assume that $ J^*\ge  J(x_{r_k})$. 
Notice that $\OAP_{r_k-1}$ redefined at line \ref{inexact: inexact restart} of the Algorithm \ref{alg: inexact}, is
$$  \OAP_{r_k-1}=  \{ x \in \mathbb{R}^n \mid J(x)=\|x\|^2 \geq J(x_{r_k})\ \}.   $$
Hence, $\A \subset \OAP_{r_k}$ 
since, for every $x\in \A$, $ J(x)\ge J^*\ge J(x_{r_k})$.

To complete the proof, let us show that assuming $ J^*<  J(x_{r_k})$ leads to a contradiction.
Since $J^*\ge 0$, $r_0 = 0$, and $J(x_{0})=0$, there exist $(s,q)\in \mathbb{K}^2$ such that 
$$ J(x_{0})\le J(x_{s}) \le J^*<   J(x_{s})+\varepsilon= J(x_{q})\le  J(x_{r_k}).   $$

If $J(x_{q})<  J(x_{r_k})$, which means that $r_k\neq q$,
then there exists a next restart iteration $\ell \le r_k$ after $q$.
If $r_k = q$, the existence of $\ell$ follows from the assumption
that $r_k\neq \max \mathbb{K}$.

By construction, the points $x_{q+1},\ldots, x_{\ell-1} $ are computed by Subroutine \ref{alg: outsolver}.1. 
By Lemma \ref{lem: sub1 outcome}, $J(x_{q})=J(x_{q+1})=\ldots=J(x_{\ell-1} )$.
Then, set  $\OAP_{\ell-1}$ built in  \eqref{e:recurOk} (line \ref{line4alg3} of Algorithm \ref{alg: inexact}) is expressed as
\begin{equation}
    \label{oapline4alg3}
    \begin{aligned}
    \OAP_{\ell-1}= & \{ x \in \mathbb{R}^n \mid \|x\|^2 \geq J(x_{q}) \}\cap\\
    &\{ x \in \mathbb{R}^n \mid 
    (\forall p\in \{q,q+1,\ldots, \ell-1\})\, (\forall i \in I(x_p)) \,
    L_i(\varphi_i^p, x_p)(x) \le 0\}. 
\end{aligned}
\end{equation}
According to \eqref{quad form constr}, we have
\begin{equation}
    \label{final1}
    \A \subset \{ x \in \mathbb{R}^n \mid 
    (\forall p\in \{q,q+1,\ldots, \ell-1\})\, (\forall i \in I(x_p)) \,
    L_i(\varphi_i^p, x_p)(x) \le 0\}.
\end{equation}

By Assumption \ref{assumption:epsilon}, $\A \cap \operatorname{lev}_{=J(x_{q})} J\neq \varnothing$. It then follows from
\eqref{final1} that
there exists a point $\overline{x}\in \A$, with $J(\overline{x})=J(x_{q})$ which is a global minimizer of $J$ over $\OAP_{\ell-1}$.
This contradicts the fact that Subroutine \ref{alg: outsolver}.2 is performed at iteration $\ell-1$ (since $\ell\in \mathbb{K}$). Indeed, by Lemma \ref{lem: sub1 outcome}, the optimal value $J_\ell= \min_{x\in \OAP_{\ell-1}}J(x)$  should then satisfy $J_\ell >J(x_{\ell-1})=J(x_{q})$.

\end{proof}

The following two propositions state the main properties of Inexact cutting spheres Algorithm \ref{alg: inexact}.

\begin{proposition}
\label{th: inexact at feasibility}
Consider the Inexact cutting sphere algorithm \ref{alg: inexact} under Assumption \ref{assumption:epsilon}.
Assume that Algorithm  \ref{alg: inexact} terminates at iterate $\bar{k}$ such that $I(x_{\bar{k}})=\varnothing$. Then, 
Iterate ${x}_{\overline{k}}$ is an $\varepsilon$-solution to Problem \eqref{problem}.
\end{proposition}
\begin{proof}
Algorithm  \ref{alg: inexact} returns a point ${x}_{\overline{k}}\in \X$, at an iterate $\overline{k}\in \mathbb{N}$.
By assumption, $I({x}_{\overline{k}})=\varnothing$, i.e., ${x}_{\overline{k}}$ is feasible for Problem \eqref{problem}. Since ${x}_{\overline{k}}\in \mathcal{A}$, we have that $J(x_{\bar{k}})\geq J^*$. 
The following two cases may arise:
\begin{enumerate}
    \item $J(x_{\bar{k}})=J^*$, then $x_{\overline{k}}$ is a global solution to Problem \eqref{problem}, hence a trivial $\varepsilon$-solution to this problem.
    \item Otherwise, $J(x_{\bar{k}})> J^*$. In this case the previous restart arisen at $r_{\bar{k}}$  was updated by Subroutine~\ref{alg: outsolver}.2.
    We have  $J(x_{\bar{k}})=J(x_{r_{\bar{k}}})$. Therefore $r_{\bar{k}}> 0$
    since $J(x_0) = 0 \le J^*$, and $J(x_{\bar{k}})
    =J(x_{r_{\bar{k}}-1})+\varepsilon$. Since there was necessarily a restart at some iteration $\ell \in  \{0,\dots,r_{\bar{k}}-1\}$, it follows from Corollary \ref{corollarytolemma} that $$J(x_{r_{\bar{k}}-1})+\varepsilon=J(x_{\bar{k}})> J^*\ge J(x_\ell) = J(x_{r_{\bar{k}}-1}).$$ 
    Equivalently,  
    $$J(x_{{\bar{k}}})> J^*\ge J(x_{{\bar{k}}})-\varepsilon,
    $$
    and we deduce  that $J^*+\varepsilon \ge J(x_{\bar{k}})$.
  Hence, $x_{\bar{k}}$ is an $\varepsilon$-solution to problem \eqref{problem}.
  
\end{enumerate}

\end{proof}

\begin{proposition}\label{prop:inexactbis} 
Consider Problem \eqref{problem}.
Suppose that Assumptions \ref{a:compact} and \ref{assumption:epsilon} hold. 
Let $(\mathcal{O}_{k-1},x_k)_{k\in  \mathbb{L}}$ be generated by Algorithm \ref{alg: inexact}.
\begin{enumerate}
    \item\label{prop:inexi} 
    For every $k\in \mathbb{L}$, we have
    $\mathcal{O}_{k-1} \subseteq \mathcal{O}_{k-2} \subseteq \ldots \subseteq \mathcal{O}_{r_{k}-1}$.
    \item\label{prop:inex3} $(J(x_k))_{k\in \mathbb{L}}$ is a nondecreasing sequence bounded from above
    by $J^*+\varepsilon$ and $(J(x_k))_{k\in \mathbb{K}}$ is an increasing sequence.
    \item\label{prop:inexii}  For every $k\in \mathbb{L}$,
    $
    \{\mathcal{A}\setminus \lev{J^*+\varepsilon}J\}\subseteq \mathcal{O}_{k-1}$.
    \item\label{prop:inexiii}  For every $k\in \mathbb{L}$, $I(x_k) = \varnothing$ if and only if $x_k \in {\mathcal{A}}$ and $J(x_k)\le J^*+\varepsilon$.
    
    \item For every $k\in \mathbb{L}$, the number of constraints in  $\mathcal{O}_{k-1}$ never exceeds $\overline{m}$.
    \item There exists an iteration ${k^*}\in \mathbb{N}$ for which the algorithm stops.
    \item Either 
    $x_{{k^*}} \in \lev{J^*+\varepsilon}J$ and $x_{{k^*}} \in \A$ , or 
    the resulting lower approximation to $J^*+\varepsilon$ is
    $J(x_{{k^*}})$.
  
\end{enumerate}
\end{proposition}

\begin{proof}\ 
\begin{enumerate} 
    \item These inclusions follow from the fact that, between two consecutive restarts, Algorithm~\ref{alg: warm_rest} accumulates constraints.
       \item  In a cumulative step, the non-decreasing property follows from
    \eqref{prop:warmi}. In turn, the cost increases when a restart iteration
    occurs at iteration $k$ as Subroutine \ref{alg: outsolver}.2 (inexact restart iteration) yields
    $$x_{k+1}\in\{x\in \mathbb{R}^n\ |\ J(x)=J(x_k)+\varepsilon\}.$$
    This shows that $(J(x_k))_{k\in \mathbb{K}}$ is an increasing sequence.\\
    Let $k\in \mathbb{L}$.
    If $r_k\neq 0$, let $\ell \in \mathbb{K}$ be the restart iteration before $r_k$. 
    It follows from Corollary~\ref{corollarytolemma} that $J(x_\ell)\le J^*$
    and 
    \begin{equation}\label{e:JkJrkJse}
    J(x_k)=J(x_{r_k})=J(x_\ell)+\varepsilon\le J^*+\varepsilon.
    \end{equation}
    If $r_0 = 0$, the inequality still holds, since $J(x_k) = J(0) \le J^*$.
    \item Notice that, for every $k\in \mathbb{L}$, 
    \eqref{oapline4alg3} and \eqref{final1} hold if $\ell=k$ and
    $q = r_k$. 
    Moreover, since \eqref{e:JkJrkJse} is satisfied,
     $$
     \{x \in \mathbb{R}^n \mid J(x)> J^*+\varepsilon\}\subset \{x \in \mathbb{R}^n \mid J(x)\ge J(x_{r_k})\}.
     $$
     By \eqref{quad form constr} 
     $$
     \A\subseteq \{ x \in \mathbb{R}^n \mid 
   (\forall\, r_k\le p\le k-1),\, (\forall i \in I(x_p))\;
    L_i(\varphi_i^p, x_p)(x) \le 0\}.
     $$
     Therefore, we have
        \begin{align*}
     &\mathcal{A}\setminus (\lev{J^*+\varepsilon}J)\\
     &\subseteq \{ x \in \mathbb{R}^n \mid 
   (\forall\, p\in\{r_k,\ldots,k-1\})\,(\forall i \in I(x_p))\;
    L_i(\varphi_i^p, x_p)(x) \le 0\}
    \cap \{x \in \mathbb{R}^n \mid J(x)> J^*+\varepsilon\}\\
    &\subset \{ x \in \mathbb{R}^n \mid 
   (\forall\, p\in\{r_k,\ldots,k-1\})\,(\forall i \in I(x_p))\;
    L_i(\varphi_i^p, x_p)(x) \le 0\}
    \cap \{x \in \mathbb{R}^n \mid J(x)\ge J(x_{r_k})\}\\&=
    \OAP_{k-1}.
    \end{align*}
     \item\label{prop:warmiiii.2} This follows from Proposition \ref{th: inexact at feasibility} and \eqref{prop:inex3}.
\item This property is ensured by the construction of Algorithm \ref{alg: inexact}. 

\item  
%

Since $\varepsilon>0$, we repeat \emph{mutatis mutandis} the proof of Proposition \ref{prop:warmbis}.v which shows that the number of restart iterations must be finite, i.e., the cardinality of set $\mathbb{K}$ is finite. 
\item 
The proof follows from 
Proposition \ref{th: inexact at feasibility} and \eqref{prop:inex3}.


    \end{enumerate}
    \end{proof}

\section{First numerical example: the circular packing problem}
\label{sec packing}
The circular packing problem (\ref{cpp}), introduced in the next Section \ref{subsec pack ref}, is an NP-Hard problem and it is often reformulated as a QCQP problem. 
An extensive literature on \ref{cpp} exits (see, e.g., \cite{ addis2008efficiently, castillo2008solving,stetsyuk2016global})\footnote{The websites '\url{http://www.packomania.com/cciuneq/}' and '\url{http://hydra.nat.uni-magdeburg.de/packing/ccis/ccis.html}' are useful resources.}.

Our aim is to test the Inexact cutting sphere algorithm \ref{alg: inexact} on a simple instance of \ref{cpp} for illustrative purposes.
\begin{itemize}
    \item In Section \ref{subsec pack ref},  we define \ref{cpp} and reformulate it in the form of \eqref{problem}.
    \item In Section \ref{subsec numerical}, we provide some numerical results.
    More information on how we solve the outer approximation problem \eqref{outerproblem} at each iteration can be found in Subsection \ref{subsec compute}.
\end{itemize}

\subsection{Reformulating the packing problem in the form of (P)}
\label{subsec pack ref}

The circular packing problem consists in finding the minimal radius $r$ of a circle in the plane, centered at zero, that contains $m$ circles of given radius $r_i$, for $i\in \{1,\dots,m\}$.
The centers $(x_i,y_i)_{1\le i \le m}$ of the circles to be packed and the radius $r$ of the packing circle 
are the sought variables of (CPP).

Let $x=(x_i)_{1\le i \le m}$ and $y=(y_i)_{1\le i \le m}$.
The circular packing problem can be modelled as follows:
\begin{equation}
    \label{cpp}
    \tag{CPP}
     \begin{aligned}
        &\minimize{r,x,y}\quad r^2\\
        &\text{s.t.}\;
        \begin{cases}
        f_i(x,y,r):=x_i^2+y_i^2- (r-r_i)^2\le 0,\;\;i \in \{1,\ldots,m\}\\
        f_{i,j}(x,y):=-(x_i-x_j)^2-(y_i-y_j)^2+(r_i+r_j)^2\le 0,\;\;1\le i<j\le m\\
        r\ge r_i,\;\; i \in \{1,\ldots,m\}.
        \end{cases}
    \end{aligned}
\end{equation}
   
\eqref{cpp} is a special case of Problem \eqref{prob: P0}. 
Hence, \eqref{cpp} can be cast in the form of \eqref{problem} as discussed in Section \ref{chap: reformulation}. 
The calculations presented in Section \ref{chap: reformulation} can be performed as follows.
   
The first and the last group of constraints can be rewritten as
    \begin{align*}
    (\forall i \in \{1,\ldots,m\})\;
        \begin{cases}
        f_i(x,y,r):=x_i^2+y_i^2- (r-r_i)^2\le 0,\\
        r\ge r_i
        \end{cases}
        \quad \Leftrightarrow \quad
        & r\geq F(x,y)
    \end{align*}
where, for every $i\in\{1,\ldots,m\}$, $\|(x_i,y_i)\|$ is the Euclidean norm of point $(x_i,y_i)\in\mathbb{R}^2$ and we have defined
\[
F(x,y)=\max_{i\in \{1,\dots,m\}} \{  (\|(x_i,y_i)\|+r_i)^2\}
\]
We can thus rewrite the packing problem as 
 \begin{align}
 \begin{aligned}\label{prob: packing}
        &\minimize{x,y}\;F(x,y),\\
        &\text{s.t.}\ \ 
        f_{i,j}(x,y)\le 0,\ \ 1\le i<j\le m.
         \end{aligned}
    \end{align}
As described in Section \ref{chap: reformulation},  we need to choose $(\rho,\eta)\in]0,+\infty[ \times \mathbb{R}$, such that
\begin{equation}
    \label{set: D_cpp}
    D = \left\{(\widehat{x},\widehat{y})\in  \operatorname{Argmin}_{(x,y)\in S}F(x,y) \mid F(\widehat{x},\widehat{y}) + \eta \ge \frac{\rho}{2}\|(\widehat{x},\widehat{y})\|^2\right\}\neq\varnothing. 
\end{equation}   
Take $\rho/2=1/m$ and $\eta=0$. 
For every $(x,y)\in \mathbb{R}^{2m}$,
$$
F(x,y)\ge \max_{i\in \{1,\ldots,m\}}\,(\{\|({x}_i,{y}_i)\|^2\}\ge \frac{1}{m}\sum\limits_{i=1}^m \|({x}_i,{y}_i)\|^2=\frac{1}{m}\|({x},{y})\|^2,
$$
where $\|({x},{y})\|$ denotes the Euclidean norm of $({x},{y})\in \mathbb{R}^{2m}$. 
According to Theorem \ref{th: minnoncon},  introducing an auxiliary variable $p\in\mathbb{R}$ allows us to rewrite \eqref{prob: packing} equivalently as
\begin{equation}
    \label{cpp new}
    \tag{CPP.2}
    \begin{aligned}
         &\minimize{(x,y,p)\in\mathbb{R}^{2m+1}}\quad \|(x,y,p)\|^2\\
        &\text{s.t.}\;
        \begin{cases}
        f_{i,j}(x,y)\le 0,\ \ 1\le i<j\le m,\\
        \displaystyle g(x,y,p):=F(x,y)-\frac{\|(x,y,p)\|^2}{m}\le 0.
        \end{cases}
    \end{aligned}
\end{equation}

Once a solution $(x^*,y^*,p^*)$ of
\eqref{cpp new} is found, we get the optimal locations $(x^*,y^*)$ of the centers of the circles to be packed in the plane. 
The minimal radius $r^*$ of the circle centered at the origin that contains the circles of radius $(r_i)_{1\le i \le m}$ satisfies
$r^*=F(x^*,y^*)$.

\subsection{Numerical experiment with four circles of radius 1}
\label{subsec numerical}

Our experiment consists in an instance of the packing problem \eqref{cpp new}, where we pack four circles of radius one. 
Hence, the dimension of the search space is $9$ and the number of weakly convex (quadratic) constraints is $m(m-1)/2+1=7$.

Since this is a simple instance, it is not difficult to find an optimal position of the circle centers in the plane, which can be 
\begin{equation}
    \label{xopt}
    (x^*,y^*)=(x_1,x_2,x_3,x_4,y_1,y_2,y_3,y_4)=(1,1, -1, -1,  1   , -1   ,  1   , -1).
\end{equation}

Moreover, we can check, with a grid search, that the variable $p\in \R$, which is essential in the formulation of the circular packing problem \eqref{cpp new}, takes the optimal value $p^*=3.914$ for $(x^*,y^*)$ as in \eqref{xopt}.
Therefore, the optimal level set for \eqref{cpp new} is $23.3194$.

We start from a random point on $\operatorname{lev}_{=21}\,J$
and we set $\varepsilon=3$. 

Then, we expect the Inexact cutting sphere algorithm \ref{alg: inexact} to  
\begin{enumerate}
    \item establish that no optimal feasible solution for the  \eqref{cpp new} exists at $\operatorname{lev}_{=21}\,J$,
    \item  jump to $\operatorname{lev}_{=24}\,J$ with a restart iteration,
    \item find a feasible solution $(x^*,y^*,p^*)$ for  \eqref{cpp new} at the current level set, i.e. $J(x^*,y^*,p^*)=\|(x^*,y^*,p^*)\|^2=24$.
\end{enumerate}

We ran our experiment on a	Microsoft Windows 10 Pro machine with 128 GB of RAM, type DDR4, supported by a 3.5 Ghz processor AMD Ryzen 9 3950X 16-Cores.

After $4$ hours, $12$ minutes, and $20$ seconds, at iteration $1143$, Algorithm \ref{alg: inexact} establishes that no feasible solution for \eqref{cpp new} exists at $\operatorname{lev}_{=21}\,J$. The number of constraints at this point is $5231$. So a restart iteration happens and $x_{1143}$ is a random point on the level set  $\operatorname{lev}_{=24}\,J$. 
Then, after $6$ minutes, for a total running time of 4:18:01 h, at iteration number $2267$, the feasible $\varepsilon$-optimal solution 
\begin{align*}
    &(x^*,y^*,p^*)=( 1.33713909,  0.58961196, -1.31825987, -0.5798902 ,\\
    & -0.55018788,
        1.32148933,  0.59123494, -1.32301583,  3.95491652)
\end{align*}
is found, which is graphically represented in Figure \ref{4cpp}. It is worth mentioning that for level sets which are far from optimal level set, Algorithm \ref{alg: inexact} is much faster to recognize  that no feasible solution for \eqref{problem} belongs to the level set and to perform a restart iteration.

\begin{figure}[H]
    \centering
    \includegraphics[width=5cm,height=5cm]{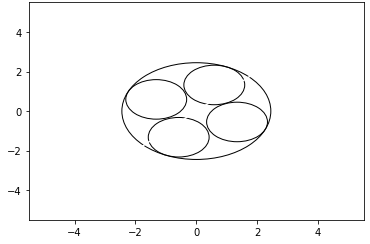}
    \caption{$(x^*,y^*,p^*)$ for the  \eqref{cpp new}  with four circles of radius one. }
    \label{4cpp}
\end{figure}

Computational details for solving the outer approximation problem \eqref{outerproblem} at each iteration are given next.

\subsection{ Implementing
Subroutine \ref{alg: outsolver}  }
\label{subsec compute}

The implementation of 
Subroutine \ref{alg: outsolver}, which is performed at each step of the Inexact cutting sphere algorithm \eqref{alg: inexact} in order to find $(x,y,p)_{k+1}$, is as follows:
\begin{enumerate}
\item We find a point $(\hat{x}_1,\hat{y}_1,\hat{p}_1)\in \OAP_k^\alpha$, where $\OAP_k^{\alpha_k}$  is the polyhedron defined in \eqref{Feasibility_prob} and $\alpha_k=\|(x_k,y_k,p_k)\|^2$, such that $\|(\hat{x}_{1},\hat{y}_{1},\hat{p}_{1}\|^2\le \alpha_k$.
\item The solver \cite{gurobi} is employed to solve the concave minimization problem:
\begin{equation}
    \label{concaveprob}
    \maximize{(x,y,p)\in\OAP_k^{\alpha_k}}\, \|(x,y,p)\|^2
\end{equation}
with an early stopping command, such that the solver \cite{gurobi} terminates as soon as it finds a point $(\hat{x}_2,\hat{y}_2,\hat{p}_2)$, feasible for $\OAP_k^{\alpha_k}$ such that $\|(\hat{x}_2,\hat{y}_2,\hat{p}_2)\|^2 \ge \alpha_k $.
\item If such a point $(\hat{x}_2,\hat{y}_2,\hat{p}_2)$ exists, as already discussed in Subsection \ref{sec: feas quad }, we can find a global solution $(x_{k+1},y_{k+1},p_{k+1})$ for the outer approximation problem \eqref{outerproblem} by looking for the point on the segment linking 
 $(\hat{x}_1,\hat{y}_1,\hat{p}_1)$ to $(\hat{x}_2,\hat{y}_2,\hat{p}_2)$ such that 
$\|(x_{k+1},y_{k+1},p_{k+1})\|^2=\alpha_k$.
This provides the result of 
Subroutine \ref{alg: outsolver}.1.
\item Otherwise, \cite{gurobi} establishes that the optimal solution $(\overline{x},\overline{y},\overline{p})$ of \eqref{concaveprob} has an objective value $\|\overline{x},\overline{y},\overline{p}\|^2 < \alpha_k $ and we change level set by setting $(x_{k+1},y_{k+1},p_{k+1})$ equal to a random point on the level set $\alpha_k+\varepsilon$, 
as described in Subroutine \ref{alg: outsolver}.2. 
\end{enumerate}

\section{Second numerical experiment: Multiclass Neyman-Pearson Classification}
\label{sec multiclass}

In this section, we apply our method to the  Multiclass Neyman-Pearson Classification (MNPC) problem, which involves a weakly convex objective function and weakly convex constraints. We consider a simple instance of the MNPC problem and we find a feasible global $\epsilon$-optimal solution.

\subsection{Reformulation of the Neyman-Pearson Classification problem as \eqref{problem}}
We formulate the MNPC problem as in \cite{ma2020quadratically}. 
Consider a training dataset formed by a number of instances  
belonging to $K$ classes indexed by $k\in \{1,\dots,K\}$.
For every $k\in \{1,\ldots,K\}$, let $\mathbb{D}_k$ be sets of indices of 
instances in the class $k$, that is every instance of the training dataset $\zeta_t\in \mathbb{R}^p$ belongs to class $k$ if and only if $t\in \mathbb{D}_k$.
Let $|\mathbb{D}_k|$, $k\in \{1,\ldots,K\}$, be the cardinality of each class, i.e. the number of instances $\zeta_t$ with $t\in \mathbb{D}_k$.

We want to find the optimal weight vectors $x_k\in \mathbb{R}^p$, with $k\in \{1,\ldots,K\}$, of a linear classifier.  A new instance
$\overline{\zeta}\in \mathbb{R}^p$ is classified by searching
\begin{equation}
    \max_{k\in \{1,\ldots,K\}}\ \ x_k^\top\overline{\zeta}.
\end{equation}
In the following, the optimization variable will be the vector $x\in \mathbb{R}^{n}$
with $n=pK$ obtained by concatenating the vectors $(x_k)_{1\le k \le K}$. 

To achieve a good classification accuracy by prioritizing the class $k=1$ and simultaneously controlling the classes $k\in \{2,\ldots,K\}$, we solve the following problem:
\begin{equation}
\label{prob: NPclass}
\tag{NPC}
\begin{aligned}
    &\minimize{x\in \mathbb{R}^{n}}\ \ F_1(x):=\frac{1}{|\mathbb{D}_1|} \sum\limits_{\ell \neq 1}\sum\limits_{t\in \mathbb{D}_1} \, \psi(x_1^\top \zeta_t - x_\ell ^\top\zeta_t)\\
    &\text{s.t}\quad
    \begin{cases}
    &\displaystyle   \frac{1}{|\mathbb{D}_k|} \sum\limits_{\ell\neq k}\sum\limits_{t\in \mathbb{D}_k} \, \psi(x_k^\top\zeta_t - x_\ell^\top\zeta_t)\le r_k,\ \ k\in \{2,\ldots,K\},\\
    &\|x_k\|^2\le \lambda^2,\ \ k\in \{1,\ldots,K\},
    \end{cases}
    \end{aligned}
\end{equation}
where $\lambda$ and $(r_k)_{2\le k \le K}$ are positive real constants 
and
$$
(\forall \xi \in \R) \quad \psi(\xi) = (1+e^\xi)^{-1}.
$$
The first, second, and third order derivatives of $\psi$ are given by
\begin{align*}
(\forall \xi \in \R) \quad  \begin{cases}
&\psi'(\xi) = -e^\xi(1+e^\xi)^{-2}\\
&\psi''(\xi) = e^\xi (e^\xi-1) (1+e^\xi)^{-3}\\
&\psi'''(\xi) = -e^\xi (e^{2\xi}-4 e^\xi +1) (1+e^\xi)^{-4}.
\end{cases}
\end{align*}
Since $\psi$ is a function with a bounded second-order derivative,
it is $\rho_\psi$-weakly convex \cite[Theorem~18.15(iii)]{bauschke2017convex}. Its minimum is reached when $e^\xi = 2-\sqrt{3}$
which is the smaller zero of $\psi'''$ and we have $\psi''(2-\sqrt{3}) = -1/(6\sqrt{3})= -\rho_\psi$.

Since, for every $(k,\ell) \in \{1,\ldots,K\}^2$ and $t\in\mathbb{D}_k$
\[
x \mapsto \psi(x_k^\top\zeta_t - x_\ell^\top\zeta_t)
\]
is the composition of a linear operator and a weakly convex function, it is weakly convex. We deduce that, for every $k\in \{1,\ldots,K\}$,
\[
x \mapsto \frac{1}{|\mathbb{D}_k|} \sum\limits_{\ell\neq k}\sum\limits_{t\in \mathbb{D}_k} \, \psi(x_k^\top\zeta_t - x_\ell^\top\zeta_t)
\]
is also $\rho_k$-weakly convex. The weak-convexity moduli
can be estimated
as shown in the following remark.
\begin{remark}
\label{rem: NPC get rho}
Without loss of generality, we can focus on $\rho_1$. First we calculate the Jacobian of $F_1$ at $x\in \R^n$:
\begin{equation}
\label{eq: nablaF1}
    \nabla F_1(x)= 
    \frac{1}{|\mathbb{D}_1|}\sum\limits_{\ell\neq 1}\sum\limits_{t\in \mathbb{D}_1}  
    \psi'(x_1^\top\zeta_t - x_\ell^\top\zeta_t)\,
    u_\ell \otimes \zeta_t,
\end{equation}
where $\otimes$ is the Kronecker product and, for every $\ell \in \{2,\ldots,K\}$, $u_\ell$ is the $K$-dimensional vector whose first component is $1$,
$\ell$-th component is -1, and all the other components are zero.
We deduce the Hessian of $F_1$ at $x$:
\begin{equation}
\label{eq: hessF1}
    \nabla^2 F_1(x)=
    \frac{1}{|\mathbb{D}_1|}\sum\limits_{\ell\neq 1} \sum\limits_{t\in \mathbb{D}_1}  
    \psi''(x_1^\top\zeta_t - x_\ell^\top\zeta_t)\,
    (u_\ell u_\ell^\top)  \otimes (\zeta_t \zeta_t^\top).
\end{equation}
It follows that an estimate of the weak-convexity modulus of $F_1$ is
\[
\rho_1 = \frac{\rho_\psi}{|\mathbb{D}_1|} \sum\limits_{\ell\neq 1}\sum\limits_{t\in \mathbb{D}_1}  
 \|u_\ell\|^2 \|\zeta_t\|^2
 = \frac{2(K-1)}{|\mathbb{D}_1|}\rho_\psi \sum\limits_{t\in \mathbb{D}_1}  
 \|\zeta_t\|^2.
\]
\end{remark}

Consequently, \eqref{prob: NPclass} is an instance of \eqref{prob: P0}.
Let us reformulate \eqref{prob: NPclass} as an instance of \eqref{problem}, Let ${x}^*\in \mathbb{R}^{n}$ be a global optimum  of \eqref{prob: NPclass}. 
Since $x^*$ is feasible for  \eqref{prob: NPclass}, the constraints $\lambda^2 \ge \|x_k^*\|^2$ are satisfied, for every $k\in \{1,\ldots,K\}$, we have
\begin{equation}
\label{eq: reformulation}
     F_1(x^*)+K\lambda^2\ge K\lambda^2\ge \sum\limits_{k=1}^K\|x_k^*\|^2=\|x^*\|^2.
\end{equation}
We can apply Theorem \ref{th: minnoncon} by choosing $\eta= K\lambda^2$ and $\rho=2$, and reformulate \eqref{prob: NPclass} as follows.
\begin{equation}
\label{prob: NPclass refor}
\tag{$\text{NPC}'$}
\begin{aligned}
    &\minimize{x\in \mathbb{R}^{n+1}}\ \ \|x\|^2\\
    &\text{s.t.}\quad 
    \begin{cases}
    \displaystyle\frac{1}{|\mathbb{D}_1|} \sum\limits_{\ell\neq 1}\sum\limits_{t\in \mathbb{D}_1} \, \psi(x_1^\top\zeta_t - x_\ell^\top\zeta_t)+ K\lambda^2- \|x\|^2 
    \le 0\\
    \displaystyle\frac{1}{|\mathbb{D}_k|} \sum\limits_{\ell\neq k}\sum\limits_{t\in \mathbb{D}_k} \, \psi(x_k^\top\zeta_t - x_\ell^\top\zeta_t)-r_k \le 0,\ \ k\in \{2,\ldots,K\}\\
    \displaystyle\|x_k\|^2\le \lambda^2,\ \ k\in \{1,\ldots,K\}.
    \end{cases}
    \end{aligned}
\end{equation}

\subsection{Numerical experiment}

We consider the database Iris from LIBSVM, which includes three classes of fifty instances each (i.e., for every 
$k\in\{1,2,3\}$, $|\mathbb{D}_k|=50$), with the feature space of
dimension $p=4$.

We consider problem \eqref{prob: NPclass} with hyperparameters $\lambda=0.3$, $r_2=r_3=0.92$.
The choice of the hyperparameters is 
consistent, in terms of magnitude, with the choice of hyperparameters in \eqref{prob: NPclass} made in \cite{ma2020quadratically} for different LIBSVM datasets.
We reformulate \eqref{prob: NPclass} as \eqref{prob: NPclass refor} by setting $\eta=K\lambda^2=0.27$ and $\rho=2$. 

We run the algorithm by choosing $\varepsilon\in \{0.5,0.2,0.05\}$. The following solutions are provided
by the algorithm:

\begin{itemize}
\item if $\varepsilon=0.5$,
\begin{align*}
    x^*=&[-0.16782005,  0.18348022, -0.18064031, -0.00444046, -0.08701317,
       -0.22545917,
    \\
      &0.13552531,  0.09277094,  0.14427671,  0.14017615,
        0.20719037,  0.04599464,  0.86292273]^\top;
    \end{align*}
    \item if $\varepsilon=0.2$,
\begin{align*}
    x^*=&[-0.17051667,  0.1388078 , -0.18541628, -0.07800498, -0.06272521,
       -0.22240767, 
    \\
      &0.17923803, -0.06206627,  0.14516118,  0.17358443,
        0.06955292,  0.14535573,  0.86870314]^\top;
    \end{align*}
\item if $\varepsilon=0.05$,
\begin{align*}
    x^*=&[-0.07385115,  0.17776952, -0.14311831, -0.16719657, -0.04696297,
       -0.21780944, 
    \\
      &0.18972297,  0.07418398,  0.20093824,  0.10209327,
        0.1188008 ,  0.16037645,  0.83230311]^\top.
    \end{align*}
\end{itemize}

In Table \ref{t:MNPC}, we provide the number of iterations, the computational time in hours:minutes:seconds (with the same computer configuration as in the previous example), the value of $J$ (objective value of \eqref{prob: NPclass refor}), and the value of $F_1$ (objective value of \eqref{prob: NPclass}) at the $\varepsilon$-solutions.
\begin{table}[htb]
\begin{center}
\caption{Algorithm performance for
the MNPC problem.}
\label{t:MNPC}
\begin{tabular}{ |c|c|c|c|c| } 
\hline
$\varepsilon$ & iterations &  time & $J$ & $F_1$ \\
\hline
0.5 & 79 &0:00:03 & 1.01 & 0.7454\\
0.2& 165& 0:24:02 & 1.01& 0.7443\\
0.05& 780 & 11:10:48 &0.96 & 0.6933  \\ 
\hline
\end{tabular}
\end{center}
\end{table}

In Figure \ref{fig:plot1}, we plot the variations of $F_1$ and $J$ along the iterations for different values of $\varepsilon$.

In Table \ref{t:MNPCcl}, we provide the proportion of correctly classified samples for each of the three classes, named Iris1, Iris2, and Iris3, in the training set. It can be observed that, even for relatively large values of $\varepsilon$ the classification performance is satisfactory.

\begin{table}[htb]
\begin{center}
\caption{Classification performance.}
\label{t:MNPCcl}
\begin{tabular}{ |c|c|c|c| } 
\hline
$\varepsilon$ & Iris1 &  Iris2 & Iris3 \\
\hline
0.5 & 50/50 & 41/50& 29/50\\
0.2 & 50/50 & 41/50& 30/50 \\
0.05 & 50/50 & 41/50& 30/50\\
\hline
\end{tabular}
\end{center}
\end{table}

Other choices for the activation functions (i.e.,  $\psi(\xi)\neq (1+e^\xi)^{-1}$) and the values of $(\lambda,r_2,r_3)$ might yield better classification scores than those shown on Table \ref{t:MNPCcl}.

\begin{figure}[ht]
   \centering
   \begin{tabular}{cc}
     \includegraphics[width=0.45\textwidth]{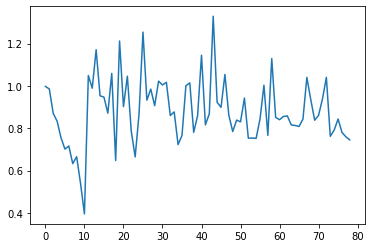}
        &\includegraphics[width=0.45\textwidth]{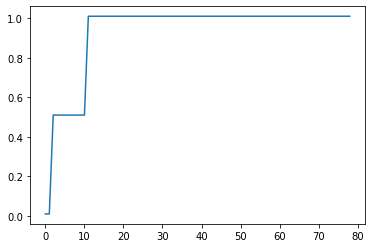}\\
        (a) & (b)\\
    \includegraphics[width=0.45\textwidth]{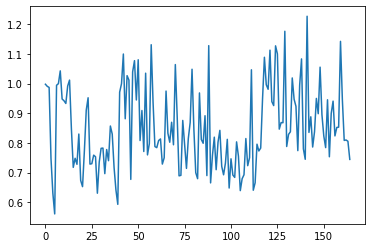}
        &\includegraphics[width=0.45\textwidth]{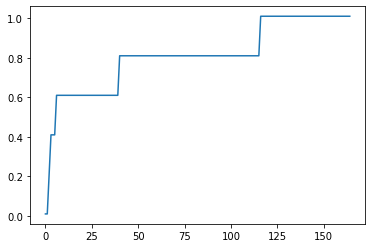}\\
        (c) & (d)\\
        \includegraphics[width=0.45\textwidth]{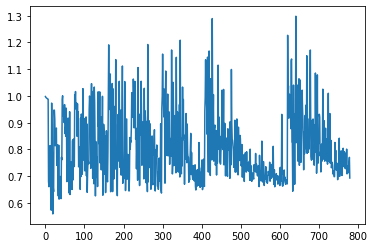}
        &\includegraphics[width=0.45\textwidth]{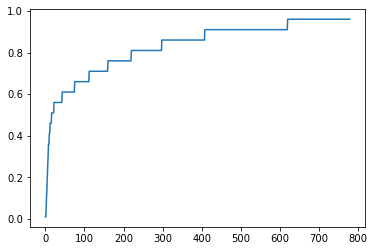}\\
        (e) & (f)\\
\end{tabular}
\caption{Variations of the objective functions $F_1$ (left column) and $J$ (right column) 
with respect to the number of iterations for 
$\varepsilon=0.5$ (a-b), 0.2 (c-d) and 0.05 (e-f).
} 
       \label{fig:plot1}
\end{figure}

\section{Conclusions}
As demonstrated by Theorem \ref{th: minnoncon}, Problem \eqref{problem} covers a broad spectrum of optimization scenarios. As the main outcome of Theorem \ref{th: convergence},  Cutting sphere algorithm \ref{alg: alg1} appears as the first outer approximation algorithm that guarantees convergence to a global solution to \eqref{problem}. To enhance the practicality of this algorithm, variants have been introduced in Algorithms \ref{alg: warm_rest} and \ref{alg: inexact}. Notably, implementations of the latter have shown its effectiveness
on two examples.

Future research could explore extending our convergence proofs to more general spaces and developing new techniques for constraint aggregation, aiming to further improve the performance of the proposed algorithms.

\small{\subsection*{Funding}{This work was funded by the European Union's Horizon 2020 research and innovation program under the Marie Sk{\l}odowska-Curie grant agreement No 861137. This work represents only the authors' view, and the European Commission is not responsible for any use that may be made of the information it contains.}}

\bibliographystyle{plain}
\bibliography{CopyWeakConv.bib}

\begin{thebibliography}{10}

\bibitem{addis2008efficiently}
Bernardetta Addis, Marco Locatelli, and Fabio Schoen.
\newblock Efficiently packing unequal disks in a circle.
\newblock {\em Operations Research Letters}, 36(1):37--42, 2008.

\bibitem{Atenas2023Unified}
Felipe Atenas, Claudia {Sagastizàbal}, Paulo J.~S. Silva, and Mikhail Solodov.
\newblock A unified analysis of descent sequences in weakly convex
  optimization, including convergence rates for bundle methods.
\newblock {\em SIAM Journal on Optimization}, 33(1):89--115, 2023.

\bibitem{BALASHOV2010113}
Maxim~V. Balashov and Dušan Repovš.
\newblock Weakly convex sets and modulus of nonconvexity.
\newblock {\em Journal of Mathematical Analysis and Applications},
  371(1):113--127, 2010.

\bibitem{bao2011semidefinite}
Xiaowei Bao, Nikolaos~V Sahinidis, and Mohit Tawarmalani.
\newblock Semidefinite relaxations for quadratically constrained quadratic
  programming: A review and comparisons.
\newblock {\em Mathematical Programming}, 129:129--157, 2011.

\bibitem{bauschke2017convex}
Heinz~H Bauschke and Patrick~L Combettes.
\newblock {\em Convex Analysis and Monotone Operator Theory in {H}ilbert
  Spaces}.
\newblock Springer, 2017.

\bibitem{bayram2015convergence}
Ilker Bayram.
\newblock On the convergence of the iterative shrinkage/thresholding algorithm
  with a weakly convex penalty.
\newblock {\em IEEE Transactions on Signal Processing}, 64(6):1597--1608, 2015.

\bibitem{bednarczuk2023global}
Ewa~M. Bednarczuk and Giovanni Bruccola.
\newblock On global solvability of a class of possibly nonconvex qcqp problems
  in {H}ilbert spaces.
\newblock {\em Optimization}, 0(0):1--20, 2023.

\bibitem{ben2011hidden}
Aharon Ben-Tal, Dick Den~Hertog, and Monique Laurent.
\newblock {\em Hidden convexity in partially separable optimization}.
\newblock , Technical Report 2011–70. Tilburg University, Center for Economic
  Research, 2011.

\bibitem{ben1996hidden}
Aharon Ben-Tal and Marc Teboulle.
\newblock Hidden convexity in some nonconvex quadratically constrained
  quadratic programming.
\newblock {\em Mathematical Programming}, 72(1):51--63, 1996.

\bibitem{bienstock2019intersection}
Daniel Bienstock, Chen Chen, and Gonzalo Munoz.
\newblock Intersection cuts for polynomial optimization.
\newblock In {\em International Conference on Integer Programming and
  Combinatorial Optimization}, pages 72--87. Springer, 2019.

\bibitem{bohm2021variable}
Axel B{\"o}hm and Stephen~J Wright.
\newblock Variable smoothing for weakly convex composite functions.
\newblock {\em Journal of Optimization Theory and Applications}, 188:628--649,
  2021.

\bibitem{boob2023stochastic}
Digvijay Boob, Qi~Deng, and Guanghui Lan.
\newblock Stochastic first-order methods for convex and nonconvex functional
  constrained optimization.
\newblock {\em Mathematical Programming}, 197(1):215--279, 2023.

\bibitem{cannarsa2004semiconcave}
Piermarco Cannarsa and Carlo Sinestrari.
\newblock {\em Semiconcave functions, Hamilton-Jacobi equations, and optimal
  control}, volume~58.
\newblock Springer Science \& Business Media, 2004.

\bibitem{castillo2008solving}
Ignacio Castillo, Frank~J Kampas, and J{\'a}nos~D Pint{\'e}r.
\newblock Solving circle packing problems by global optimization: numerical
  results and industrial applications.
\newblock {\em European Journal of Operational Research}, 191(3):786--802,
  2008.

\bibitem{cegielski2011opial}
Andrzej Cegielski and Yair Censor.
\newblock Opial-type theorems and the common fixed point problem.
\newblock {\em Fixed-Point Algorithms for Inverse Problems in Science and
  Engineering}, pages 155--183, 2011.

\bibitem{censor2009split}
Yair Censor and Alexander Segal.
\newblock The split common fixed point problem for directed operators.
\newblock {\em J. Convex Anal}, 16(2):587--600, 2009.

\bibitem{chazelle1993optimal}
Bernard Chazelle.
\newblock An optimal convex hull algorithm in any fixed dimension.
\newblock {\em Discrete \& Computational Geometry}, 10:377--409, 1993.

\bibitem{combettes2000strong}
Patrick~L Combettes.
\newblock Strong convergence of block-iterative outer approximation methods for
  convex optimization.
\newblock {\em SIAM Journal on Control and Optimization}, 38(2):538--565, 2000.

\bibitem{combettes2001quasi}
Patrick~L Combettes.
\newblock Quasi-{F}ej{\'e}rian analysis of some optimization algorithms.
\newblock In {\em Studies in Computational Mathematics}, volume~8, pages
  115--152. Elsevier, 2001.

\bibitem{combettes2004image}
Patrick~L Combettes and J-C Pesquet.
\newblock Image restoration subject to a total variation constraint.
\newblock {\em IEEE Transactions on Image Processing}, 13(9):1213--1222, 2004.

\bibitem{cuong2022error}
Nguyen~Duy Cuong and Alexander~Y Kruger.
\newblock Error bounds revisited.
\newblock {\em Optimization}, 71(4):1021--1053, 2022.

\bibitem{davis2019stochastic}
Damek Davis and Dmitriy Drusvyatskiy.
\newblock Stochastic model-based minimization of weakly convex functions.
\newblock {\em SIAM Journal on Optimization}, 29(1):207--239, 2019.

\bibitem{drori2016optimal}
Yoel Drori and Marc Teboulle.
\newblock An optimal variant of {K}elley’s cutting-plane method.
\newblock {\em Mathematical Programming}, 160(1):321--351, 2016.

\bibitem{fabian2010error}
Marian~J Fabian, Ren{\'e} Henrion, Alexander~Y Kruger, and Ji{\v{r}}{\'\i}~V
  Outrata.
\newblock Error bounds: necessary and sufficient conditions.
\newblock {\em Set-Valued and Variational Analysis}, 18(2):121--149, 2010.

\bibitem{gurobi}
{Gurobi Optimization, LLC}.
\newblock {Gurobi Optimizer Reference Manual}, 2023.

\bibitem{jourani1996subdifferentiability}
Abderrahim Jourani.
\newblock Subdifferentiability and subdifferential monotonicity of-paraconvex
  functions.
\newblock {\em Control and Cybernetics}, 25(4), 1996.

\bibitem{kelley1960cutting}
James~E Kelley, Jr.
\newblock The cutting-plane method for solving convex programs.
\newblock {\em Journal of the Society for Industrial and Applied Mathematics},
  8(4):703--712, 1960.

\bibitem{khachiyan2009generating}
Leonid Khachiyan, Endre Boros, Konrad Borys, Vladimir Gurvich, and Khaled
  Elbassioni.
\newblock Generating all vertices of a polyhedron is hard.
\newblock {\em Twentieth Anniversary Volume: Discrete \& Computational
  Geometry}, pages 1--17, 2009.

\bibitem{khanh2023inexactproximalmethodsweakly}
Pham~Duy Khanh, Boris Mordukhovich, Vo~Thanh Phat, and Dat~Ba Tran.
\newblock Inexact proximal methods for weakly convex functions.
\newblock {\em arXiv2307.15596}, 2023.

\bibitem{kim2023equivalent}
Sunyoung Kim and Masakazu Kojima.
\newblock Equivalent sufficient conditions for global optimality of
  quadratically constrained quadratic program.
\newblock {\em arXiv2303.05874}, 2023.

\bibitem{kiwiel1990proximity}
Krzysztof~C Kiwiel.
\newblock Proximity control in bundle methods for convex nondifferentiable
  minimization.
\newblock {\em Mathematical Programming}, 46(1):105--122, 1990.

\bibitem{lewis1998error}
Adrian~S Lewis and Jong-Shi Pang.
\newblock Error bounds for convex inequality systems.
\newblock {\em Generalized convexity, generalized monotonicity: recent
  results}, 1998.

\bibitem{liu2019doa}
Qi~Liu, Yuantao Gu, and Hing~Cheung So.
\newblock {D.O.A.} estimation in impulsive noise via low-rank matrix
  approximation and weakly convex optimization.
\newblock {\em IEEE Transactions on Aerospace and Electronic Systems},
  55(6):3603--3616, 2019.

\bibitem{8955328}
Mariana Lopushanski.
\newblock Weakly convex sets in asymmetric seminormed spaces and their
  properties: Uniform continuity of multifunction intersection.
\newblock In {\em 2018 International Conference on Applied Mathematics \&
  Computer Science (ICAMCS)}, pages 10--107, 2018.

\bibitem{lukvsan1999globally}
Ladislav Luk{\v{s}}an and Jan Vl{\v{c}}ek.
\newblock Globally convergent variable metric method for convex nonsmooth
  unconstrained minimization1.
\newblock {\em Journal of Optimization Theory and Applications}, 102:593--613,
  1999.

\bibitem{ma2020quadratically}
Runchao Ma, Qihang Lin, and Tianbao Yang.
\newblock Quadratically regularized subgradient methods for weakly convex
  optimization with weakly convex constraints.
\newblock In {\em International Conference on Machine Learning}, pages
  6554--6564. PMLR, 2020.

\bibitem{mollenhoff2015primal}
Thomas M\"ollenhoff, Evgeny Strekalovskiy, Michael Moeller, and Daniel Cremers.
\newblock The primal-dual hybrid gradient method for semiconvex splittings.
\newblock {\em SIAM Journal on Imaging Sciences}, 8(2):827--857, 2015.

\bibitem{nikolova1998estimation}
Mila Nikolova.
\newblock Estimation of binary images by minimizing convex criteria.
\newblock In {\em Proceedings 1998 International Conference on Image
  Processing. ICIP98 (Cat. No. 98CB36269)}, volume~2, pages 108--112. IEEE,
  1998.

\bibitem{pallaschke2013foundations}
Diethard~Ernst Pallaschke and Stefan Rolewicz.
\newblock {\em Foundations of mathematical optimization: convex analysis
  without linearity}, volume 388.
\newblock Springer Science \& Business Media, 2013.

\bibitem{pardalos1986methods}
Panos~M Pardalos and Judah~Ben Rosen.
\newblock Methods for global concave minimization: A bibliographic survey.
\newblock {\em SIAM Review}, 28(3):367--379, 1986.

\bibitem{rafique2022weakly}
Hassan Rafique, Mingrui Liu, Qihang Lin, and Tianbao Yang.
\newblock Weakly-convex--concave min--max optimization: provable algorithms and
  applications in machine learning.
\newblock {\em Optimization Methods and Software}, 37(3):1087--1121, 2022.

\bibitem{rubinov2013abstract}
Alexander~M Rubinov.
\newblock {\em Abstract convexity and global optimization}, volume~44.
\newblock Springer Science \& Business Media, 2013.

\bibitem{selesnick2020non}
Ivan Selesnick, Alessandro Lanza, Serena Morigi, and Fiorella Sgallari.
\newblock Non-convex total variation regularization for convex denoising of
  signals.
\newblock {\em Journal of Mathematical Imaging and Vision}, 62(6):825--841,
  2020.

\bibitem{shap1}
Alexander Shapiro and Katya Scheinberg.
\newblock Duality and optimality conditions.
\newblock In {\em Handbook of semidefinite programming}, pages 67--110.
  Springer, 2000.

\bibitem{singer1997abstract}
Ivan Singer.
\newblock {\em Abstract convex analysis}.
\newblock John Wiley \& Sons, 1997.

\bibitem{stetsyuk2016global}
Petro~I Stetsyuk, Tatiana~E Romanova, and Guntram Scheithauer.
\newblock On the global minimum in a balanced circular packing problem.
\newblock {\em Optimization Letters}, 10(6):1347--1360, 2016.

\bibitem{tuy1988global}
Hoang Tuy and Nguyen Van~Thuong.
\newblock On the global minimization of a convex function under general
  nonconvex constraints.
\newblock {\em Applied Mathematics and Optimization}, 18(1):119--142, 1988.

\bibitem{Vanden1}
Lieven Vandenberghe and Stephen Boyd.
\newblock Semidefinite programming.
\newblock {\em SIAM review}, 38(1):49--95, 1996.

\bibitem{yamada2010outer}
Syuuji Yamada, Tamaki Tanaka, and Tetsuzo Tanino.
\newblock Outer approximation method incorporating a quadratic approximation
  for a dc programming problem.
\newblock {\em Journal of Optimization Theory and Applications},
  144(1):156--183, 2010.

\bibitem{zaknoon2003algorithmic}
Maroun Zaknoon.
\newblock {\em Algorithmic developments for the convex feasibility problem}.
\newblock University of Haifa (Israel), 2003.

\bibitem{zhang2019fundamental}
Tao Zhang and Zhengwei Shen.
\newblock A fundamental proof of convergence of alternating direction method of
  multipliers for weakly convex optimization.
\newblock {\em Journal of Inequalities and Applications}, 2019(1):1--21, 2019.

\end{thebibliography}

\end{document}